\documentclass[10pt]{amsart}

\usepackage[T1]{fontenc}

\usepackage[margin=3.5cm]{geometry}

\usepackage[foot]{amsaddr}

\usepackage{mathtools, amsmath, amssymb, amsthm}

\usepackage{graphicx}

\usepackage{tikz-cd}
\usepackage[all]{xy}

\usepackage{todonotes, verbatim}

\usepackage{enumerate}

\usepackage{aliascnt}

\usepackage{silence}
\usepackage[hidelinks]{hyperref}

\usepackage{csquotes}

\usepackage{bbm}

\usepackage{halloweenmath}

\makeatletter
\DeclareDocumentCommand{\newmathcommand}{ m O{0} m }{%
  \ifcsname\expandafter\@gobble\string#1\space\endcsname
    \expandafter\expandafter\expandafter\let\expandafter\csname old\string#1\expandafter\endcsname\expandafter=\csname\expandafter\@gobble\string#1\space\endcsname
  \else
    \expandafter\let\csname old\string#1\endcsname=#1
  \fi
  \expandafter\newcommand\csname new\string#1\endcsname[#2]{#3}
  \DeclareRobustCommand#1{%
    \ifmmode
      \expandafter\let\expandafter\next\csname new\string#1\endcsname
    \else
      \expandafter\let\expandafter\next\csname old\string#1\endcsname
    \fi
    \next
  }%
}
\makeatother
\newmathcommand{\r}{\mathring}

\makeatletter
\DeclareRobustCommand\bigop[2][1]{%
  \mathop{\vphantom{\sum}\mathpalette\bigop@{{#1}{#2}}}\slimits@
}
\newcommand{\bigop@}[2]{\bigop@@#1#2}
\newcommand{\bigop@@}[3]{%
  \vcenter{%
    \sbox\z@{$#1\sum$}%
    \hbox{\resizebox{\ifx#1\displaystyle#2\fi\dimexpr\ht\z@+\dp\z@}{!}{$\m@th#3$}}%
  }%
}
\makeatother
\let\originalbeef\bigstar
\renewcommand{\bigstar}{\bigop[0.8]{\originalbeef}}

\let\originalleft\left
\let\originalright\right
\renewcommand{\left}{\mathopen{}\mathclose\bgroup\originalleft}
\renewcommand{\right}{\aftergroup\egroup\originalright}

\AtBeginDocument{
\let\originalref\ref
\let\originaleqref\eqref
\renewcommand{\ref}[1]{\autoref{#1}}
\renewcommand{\eqref}[1]{{\renewcommand{\ref}{\originalref}\originaleqref{#1}}}
}

\renewcommand{\phi}{\varphi}
\renewcommand{\epsilon}{\varepsilon}

\allowdisplaybreaks

\addtocontents{toc}{\protect\setcounter{tocdepth}{1}}

\makeatletter
\newcommand{\A}{\GenericError{\@backslashchar A error}{You wrote \@backslashchar A again}{Stop writing \@backslashchar A, silly}{Mildred should stop writing \@backslashchar A}}
\makeatother

\newcommand{\tensor}{\otimes}
\newcommand{\under}{\backslash}

\DeclareMathOperator{\Sub}{Sub}
\DeclareMathOperator{\Conj}{Conj}
\DeclareMathOperator{\Parts}{Parts}
\DeclareMathOperator{\Parks}{Parks}
\DeclareMathOperator{\Marks}{Marks}

\DeclareMathOperator{\Aut}{Aut}
\DeclareMathOperator{\Stab}{Stab}
\DeclareMathOperator{\Fun}{Fun}
\DeclareMathOperator{\Tr}{Tr}
\DeclareMathOperator{\Res}{Res}
\DeclareMathOperator{\Nm}{Nm}
\DeclareMathOperator{\Cyc}{Cyc}

\DeclareMathOperator{\Pow}{Pow}

\DeclareMathOperator{\Ab}{Ab}
\DeclareMathOperator{\CGrRing}{CGrRing}
\DeclareMathOperator{\Set}{Set}
\DeclareMathOperator{\Fin}{Fin}

\DeclareMathOperator{\Pops}{Pops}

\newcommand{\op}{\textrm{op}}

\newcommand{\N}{\mathbb{N}}
\newcommand{\Z}{\mathbb{Z}}
\newcommand{\Q}{\mathbb{Q}}

\newmathcommand{\P}{\mathbb{P}}
\newcommand{\wA}{\r w}

\newmathcommand{\1}{\mathbbm{1}}

\newcommand{\divides}{\mid}
\DeclareMathOperator{\lcm}{lcm}

\newcommand{\cupdot}{\mathbin{\mathaccent\cdot\cup}}
\newcommand{\bigcupdot}{\bigop{\cupdot}}

\newcommand{\mf}[1]{(\!(#1)\!)!}

\newcommand{\num}{\sharp}
\newcommand{\dec}{\flat}

\newcommand{\ag}[1]{\left\langle#1\right\rangle}

\makeatletter
\providecommand{\leftsquigarrow}{%
  \mathrel{\mathpalette\reflect@squig\relax}%
}
\newcommand{\reflect@squig}[2]{%
  \reflectbox{$\m@th#1\rightsquigarrow$}%
}
\makeatother

\theoremstyle{definition}
\numberwithin{equation}{section}

\newaliascnt{theorem}{equation}
\newtheorem{theorem}[theorem]{Theorem}
\aliascntresetthe{theorem}

\newaliascnt{definition}{theorem}
\newtheorem{definition}[definition]{Definition}
\aliascntresetthe{definition}

\newaliascnt{proposition}{theorem}
\newtheorem{proposition}[proposition]{Proposition}
\aliascntresetthe{proposition}

\newaliascnt{corollary}{theorem}
\newtheorem{corollary}[corollary]{Corollary}
\aliascntresetthe{corollary}

\newaliascnt{lemma}{theorem}

\aliascntresetthe{lemma}

\newaliascnt{remark}{theorem}
\newtheorem{remark}[remark]{Remark}
\aliascntresetthe{remark}

\makeatletter
\def\retheorem@title{???}
\newtheoremstyle{retheorem}%
    {.5\baselineskip plus .2\baselineskip minus .2\baselineskip}
    {.5\baselineskip plus .2\baselineskip minus .2\baselineskip}
    {\normalfont}%
    {0pt}%
    {\bfseries}%
    {.}%
    {5pt plus 1pt minus 1pt}%
    {\retheorem@title}
\theoremstyle{retheorem}
\newtheorem*{@retheorem}{}
\newenvironment{retheorem}[2]{%
\def\retheorem@title{\csname#1autorefname\endcsname{} \normalfont(\ref{#2})}%
\begin{@retheorem}%
}{\end{@retheorem}\def\retheorem@title{???}}
\makeatother

\title{On the image of the total power operation for Burnside rings}

\address{University of Kentucky}
\date{\today}

\author{Nathan Cornelius}
\email{nathan.cornelius@uky.edu}

\author{Lewis Dominguez}
\email{Lewis.Dominguez@uky.edu}

\author{David Mehrle}
\email{davidm@uky.edu}

\author{Lakshay Modi}
\email{Lakshay.Modi@uky.edu}

\author{Millie Rose}
\email{math@milliero.se}

\author{Nathaniel Stapleton}
\email{nat.j.stapleton@gmail.com}

\begin{document}

\begin{abstract}
    We prove that the image of the total power operation for Burnside rings $A(G) \to A(G\wr\Sigma_n)$ lies inside a relatively small, combinatorial subring $\AA(G,n) \subseteq A(G \wr \Sigma_n)$. As $n$ varies, the subrings $\AA(G,n)$ assemble into a commutative graded ring $\AA(G)$ with a universal property: $\AA(G)$ carries the universal family of power operations out of $A(G)$. We construct character maps for $\AA(G,n)$ and give a formula for the character of the total power operation. Using $\AA(G)$, we extend the Frobenius--Wielandt homomorphism of \cite{DSY1992} to wreath products compatibly with the total power operation. Finally, we prove a generalization of Burnside's orbit counting lemma that describes the transfer map $A(G \wr \Sigma_n) \to A(\Sigma_n)$ on the subring $\AA(G,n)$. 
\end{abstract}

\maketitle

\vspace{-1em}
\tableofcontents

\newpage

\section{Introduction}

This paper explores consequences of the isomorphism of $G \wr \Sigma_n$-sets
\[
(G/H)^{\times n} \cong (G \wr \Sigma_n)/(H \wr \Sigma_n),
\]
which gives great control over the total power operation for the Burnside ring of $G$
\[
\P_n \colon A(G) \to A(G \wr \Sigma_n).
\]
The total power operation sends the isomorphism class of a finite $G$-set $X$ to the isomorphism class of the $G \wr \Sigma_n$-set $X^{\times n}$. The smallest additive summand of $A(G \wr \Sigma_n)$ generated by elements of the canonical basis and containing the image of the total power operation is a relatively small, combinatorial subring $\AA(G,n) \subseteq A(G \wr \Sigma_n)$.
With an understanding of this combinatorial structure, we prove that $\AA(G) = \bigoplus_{n \geq 0} \AA(G,n)$ inherits a commutative graded ring structure from the transfer product on $\bigoplus_{n} A(G \wr \Sigma_n)$.
Additionally, we show that the construction $\AA(G)$ is functorial in the underlying abelian group of $A(G)$.
With this in hand, we 
\begin{enumerate}[A.)]
    \item Show that $\AA(G)$ carries the universal family of power operations out of $A(G)$.
    \item Describe a character map for $\AA(G,n)$.
    \item Give a formula for the character of $\P_n$. 
    \item Explicitly describe how any map of abelian groups $A(H) \to A(G)$ induces a map $\AA(H) \to \AA(G)$ and give a formula for the character of this map.
\end{enumerate}
This last item gives us the ability to extend the Frobenius--Wielandt map of \cite[Theorem 1]{DSY1992} to the Burnside ring of the wreath product and also to give an elegant formula, in the spirit of Burnside's orbit counting lemma, for the transfer map $A(G \wr \Sigma_n) \to A(\Sigma_n)$ along the quotient $G \wr \Sigma_n \to \Sigma_n$ after restriction to $\AA(G,n) \subseteq A(G \wr \Sigma_n)$. 

To elaborate, let $G$ be a finite group and let $A(G)$ be the Burnside ring of $G$. 
This is a free $\Z$-module of rank equal to the number of isomorphism classes of transitive $G$-sets.
The Burnside ring admits a ``character map'' landing in $\Marks(G) = \Fun(\Conj(G), \Z)$, where $\Conj(G)$ denotes the set of conjugacy classes of subgroups of $G$.
The character map is an isomorphism after rationalization.
The total power operation is a multiplicative map
\[
\P_n \colon A(G) \to A(G \wr \Sigma_n)
\]
induced by taking a $G$-set $X$ to the $n$-fold Cartesian product $X^{\times n}$ equipped with the usual $G \wr \Sigma_n$-action.

Let $\AA(G,n) \subseteq A(G \wr \Sigma_n)$ be as defined above. It turns out that this subset of $A(G \wr \Sigma_n)$ can be described in many more ways.
We provide three further descriptions here.

A finite $G \wr \Sigma_n$-set $Y$ is called submissive if $Y$ admits an equivariant embedding into $X^{\times n}$ for some $G$-set $X$.
The coproduct and product of submissive sets are submissive. In particular, we prove the following description of $\AA(G,n)$:

\begin{retheorem}{proposition}{prop:AAisKsub}
    The ring $\AA(G,n)$ is the Grothendieck ring of isomorphism classes of submissive $G \wr \Sigma_n$-sets.
\end{retheorem}

The ring $\AA(G,n)$ can also be described using the integer partitions of $n$ decorated by the set $\Conj(G)$.
Such a partition of $n$ is a function
\[
\lambda \colon \Conj(G) \times \N_{> 0} \to \N
\]
such that $\sum_{[H],m} \lambda_{[H],m} \cdot m = n$.
The quantity $\lambda_{[H],m}$ counts the number of occurrences in the partition $\lambda$ of the natural number $m$ decorated by the conjugacy class $[H] \in \Conj(G)$.
Let $\Parts(G,n)$ be the set of integer partitions of $n$ decorated by $\Conj(G)$ and $\Parks(G,n)$ be the ring of integer valued functions on $\Parts(G,n)$.

There is a ring map $\beta^* \colon \Parks(G,n) \to \Marks(G \wr \Sigma_n)$ induced by a canonical map $\beta \colon \Conj(G \wr \Sigma_n) \to \Parts(G,n)$. 
\begin{retheorem}{proposition}{cor:levelpullback}
    There is a pullback square of commutative rings
\[\begin{tikzcd}
    \AA(G,n) \ar[r] \ar[d] & A(G \wr \Sigma_n) \ar[d, "\chi"] \\
    \Parks(G,n) \ar[r, "\beta^*"] & \Marks(G \wr \Sigma_n),
\end{tikzcd}\]
where $\chi$ is the character map.    
\end{retheorem}
This provides $\AA(G,n)$ with a character map $\AA(G,n) \to \Parks(G,n)$ that is an isomorphism after rationalization.

For the third description, let $A$ be an abelian group and let $R = \bigoplus_{n \geq 0} R_n$ be a commutative $\N$-graded ring. 
An $R$-valued family of power operations on $A$ is a family of functions $(P_n \colon A \to R_n)_{n \geq 0}$ with the property that $P_0$ is the constant function $1$, $P_n(0)=0$ for $n>0$, and 
\[ \displaystyle P_n(x+y) = \sum_{i+j=n} P_i(x)P_j(y). \]
Let $\Pops(A,R)$ be the set of $R$-valued families of power operations on $A$. 
This assignment is functorial in both $A$ and $R$. 
As a functor of $R$, it is corepresentable by a commutative graded ring $\Pow(A)$. 
Using these functors, we produce the following isomorphism:
\begin{retheorem}{proposition}{cor:powang}
    There is a canonical isomorphism $\Pow(A(G)) \cong \bigoplus_n \AA(G,n)$ of commutative graded rings.     
\end{retheorem}
To explain this isomorphism further, recall that the direct sum of abelian group $A(G \wr \Sigma) = \bigoplus_n A(G \wr \Sigma_n)$ acquires a commutative graded ring structure induced by the ``transfer multiplication''
\[
X \star Y = (G \wr \Sigma_n) \times_{G \wr \Sigma_i \times G \wr \Sigma_j} (X \times Y)
\]
for $i+j = n$, $X$ a $G \wr \Sigma_i$-set, and $Y$ a $G \wr \Sigma_j$-set. 
The family of power operations $(\P_n \colon A(G) \to A(G \wr \Sigma_n))_{n \geq 0}$ gives rise to a map of commutative graded rings $\Pow(A(G)) \to A(G \wr \Sigma)$. 
This map is injective with image $\AA(G) = \bigoplus_n \AA(G,n)$.

Similarly, one can consider $RU(G \wr \Sigma) = \bigoplus_n RU(G \wr \Sigma_n)$. In this case, work of Zelevensky \cite{Zelevinsky} can be used to show that the family of power operations $(\P_n \colon RU(G) \to RU(G \wr \Sigma_n))$ gives rise to an isomorphism of commutative graded rings $\Pow(RU(G)) \xrightarrow{\cong} RU(G \wr \Sigma)$.

A map of abelian groups $F \colon A(H) \to A(G)$ induces a map 
\[ \r F \colon \AA(H,n) \cong \Pow(A(H))_n \to \Pow(A(G))_n \cong \AA(G,n). \]
Rationalization induces a linear map $\Q \otimes F \colon \Marks(H,\Q) \to \Marks(G,\Q)$ using the character isomorphism $\Q \otimes A(G) \cong \Marks(G,\Q)$, where $\Marks(G,\Q)$ denotes $\Q \tensor \Marks(G)$.
Similarly, rationalization induces a linear map $\Q \otimes \r F \colon \Parks(H,n,\Q) \to \Parks(G,n,\Q)$, where $\Parks(G,n,\Q)$ denotes $\Q \tensor \Parks(G,n)$.
\begin{retheorem}{theorem}{thm:parksform}
    Let $F \colon A(H) \to A(G)$ be a map of abelian groups, let $M$ be the matrix representing $\Q \otimes F\colon \Marks(H,\Q) \to \Marks(G,\Q)$ in the bases consisting of conjugacy classes of subgroups, and let $\r M$ be the matrix representing $\Q \otimes \r F \colon \Parks(H,n,\Q) \to \Parks(G,n,\Q)$ in the bases consisting of decorated partitions of $n$.
    We give an explicit formula for the coefficients of $\r M$ in terms of the coefficients of $M$.    
\end{retheorem}
Note that this theorem applies to any additive map -- not just to restriction and transfer maps.
Particularly interesting examples to consider are the transfer map along $G \to e$ and the Frobenius--Wielandt map.

Let $m = |G|$. The Frobenius--Wielandt map is a ring homomorphism $A(C_{m}) \to A(G)$, where $C_{m}$ is the cyclic group of order $m$. It is induced by the map $\Marks(C_{m}) \to \Marks(G)$ coming from the function $\Conj(G) \to \Conj(C_m)$ sending a conjugacy class of subgroups $[H]$ to the unique subgroup of $C_m$ of order $|H|$. 
Notably, the Frobenius--Wielandt map is not, in general, induced by a group homomorphism.

\ref{thm:parksform} together with the map $\beta$ mentioned above, allows us to extend the Frobenius--Wielandt map to wreath products:

\begin{retheorem}{proposition}{cor:FWandpower}
Let $m = |G|$ and $n>0$. We construct a Frobenius--Wielandt map
\[
w_n \colon A(C_{m} \wr \Sigma_n) \to A(G \wr \Sigma_n)
\]
compatible with the usual Frobenius--Wielandt map $A(C_m) \to A(G)$ through the total power operation.
\end{retheorem}

With this in hand, we explore the relationship between power operations and work of Romero \cite{Romero2020} that gives a group-theoretic condition under which the Frobenius--Wielandt map commutes with norms.

\subsection*{Acknowledgements} 
It is a pleasure to thank William Balderrama, Tobias Barthel, Ben Braun, Sune Precht Reeh, and Tomer Schlank for helpful discussions related to this work. We also appreciate Shahzad Kalloo's enthusiasm, ideas, and virtual participation in the Kentucky Bourbon Seminar. Other (occasional) attendees of the seminar included Evan Franchere, Prerna Dhankhar, Usman Hafeez, Joshua Peterson, Jesse Keyes.

Mehrle and Rose were supported by the eCHT under NSF Grant DMS-2135884. Stapleton was supported by a Sloan research fellowship, NSF Grant DMS-2304781, and a grant from the Simons Foundation (MP-TSM-00002836, NS).

\section{Background and notation} \label{sec:background}

Most of the material in this background section can be found in \cite{Rymer, tomdieck, DSY1992, Bouc2010}. See \cite{BMMS} for more on power operations.

\subsection{Burnside rings} \label{subsec:burnside rings}
Let $G$ be a finite group. 
Let $\Fin_G$ denote the category of finite $G$-sets and let $A^+(G) = \pi_0 {\Fin_G}^{\cong}$ denote the set of isomorphism classes of finite $G$-sets.
We denote the isomorphism class of a finite $G$-set $X$ by $[X] \in A^+(G)$.
The coproduct (disjoint union) and product of finite $G$-sets make $A^+(G)$ into a commutative semiring.

The Burnside ring $A(G)$ is the Grothendieck ring of $A^+(G)$. 
As a $\Z$-module, $A(G)$ has a canonical basis consisting of isomorphism classes of transitive finite  $G$-sets.
In this basis, the multiplication is determined by the double coset formula
\[ 
    G/H \times G/K \cong \coprod_{HgK \in H\backslash G \slash K} G/(H^g \cap K),
\]
where $H^g$ denotes the conjugation $g^{-1}Hg$.

\subsection{Restriction and transfer} \label{subsec:res and tr}
Given a group homomorphism $\phi \colon H \to G$, there are functors 
\[ \Res_\phi \colon \Fin_G \to \Fin_H \text{ and } \Tr_\phi \colon \Fin_H \to \Fin_G \]
respectively called restriction and transfer. Restriction is induced by restricting the $G$-action on a $G$-set to $H$ along $\phi$, and transfer is induced by sending an $H$-set $X$ to the $G$-set $G \times_H X$. In the literature, restriction is often called inflation when $\phi$ is surjective and transfer is often called induction when $\phi$ is injective and deflation when $\phi$ is surjective.

These functors induce maps
\[
\Res_\phi \colon A(G) \to A(H) \text{ and } \Tr_\phi \colon A(H) \to A(G).
\]
Restriction is a ring map and transfer is an $A(G)$-module map for the action of $A(G)$ on $A(H)$ through the restriction map.

\subsection{External products} \label{subsec:ext prod}
Given a finite $G$-set $X$ and a finite $H$-set $Y$, the Cartesian product $X \times Y$ is a finite $G \times H$-set.
To avoid ambiguity, we denote this operation by $X \boxtimes Y$, which we call the external product.
The external product induces a functor
\[ \boxtimes \colon \Fin_G \times \Fin_H \to \Fin_{G\times H} \]
and a homomorphism 
\[ \boxtimes \colon A(G) \tensor A(H) \to A(G \times H). \]

\subsection{Wreath products} \label{subsec:wr prod}
We write $G \wr \Sigma_n$ for the wreath product, which is the semidirect product $G^{\times n} \rtimes \Sigma_n$, where $\Sigma_n$ acts on $G^{\times n}$ on the left via permutations. 
Explicitly, for $\bar g = (\bar g_1,\ldots, \bar g_n) \in G^{\times n}$ and $\tau \in \Sigma_n$, we write $\tau\bar g$ for the element in $G^{\times n}$ with $(\tau\bar g)_i = \bar g_{\tau^{-1}(i)}$. 
Then the multiplication in $G \wr \Sigma_n$ is given by $(\bar{g},\sigma)(\bar{h},\tau) = (\bar{g}\, \sigma \bar{h}, \sigma \tau)$. 
If $G$ acts on the left on a space $X$, then $G \wr \Sigma_n$ acts on $X^{\times n}$ on the left. Explicitly, we have
\[
((\bar{g},\sigma)x)_i = \bar g_{i}x_{\sigma^{-1}(i)}.
\]

\subsection{Finite sets and orders} \label{subsec:fin sets}

For $n \in \N$, let $[n]$ denote a set of size $n$.
We identify the symmetric group $\Sigma_n$ with the group $\Sigma_{[n]} = \Aut([n])$ of bijections of the set $[n]$.
When $i+j=n$, any identification $[i] \amalg [j] \cong [n]$ induces an embedding $\Sigma_i \times \Sigma_j \hookrightarrow \Sigma_n$.
Each of these embeddings is conjugate, and we will only use these maps to induce operations on Burnside rings or similar conjugacy invariant objects, so the choice of embedding is immaterial.
More generally, for $X$ some set and functions $m,f \colon X \to \N$ where $f$ has finite support, we will consider embeddings of the form
\[ \prod_{x \in X} \Sigma_{m(x)}^{\times f(x)} \hookrightarrow \Sigma_n \text{ where } n = \sum_{x \in X} f(x)m(x). \]
Such embeddings arise from choosing any identification $\coprod_{x \in X} [f(x)] \times [m(x)] \cong [n]$.
Again, all such embeddings are conjugate, so the particular choice is immaterial.
These identifications also induce embeddings of the form
\[ G \wr \Sigma_i \times G \wr \Sigma_j \hookrightarrow G \wr \Sigma_n \text{ and } \prod_{x \in X} (G \wr \Sigma_{m(x)})^{\times f(x)} \hookrightarrow G \wr \Sigma_n, \]
each of which is conjugate to any other.

Similarly, for a set $X$, we will identify $X^{\times n}$ with $X^{[n]} = \prod_{[n]}X$.
Any identification $[i] \amalg [j] \cong [n]$ induces a bijection $X^{\times i} \times X^{\times j} \cong X^{\times n}$.
Unfortunately, the distinction between these bijections often matters:
sometimes it is necessary that the $X^{\times i}$ factor appear `first.'
Accordingly, we implicitly choose total orders on the sets $[n]$, and we put the lexicographic ordering on products and disjoint unions.
In this situation, there is exactly one order preserving isomorphism $[i] \amalg [j] \cong [n]$, and we use this one to induce `the' bijection $X^{\times i} \times X^{\times j} \cong X^{\times n}$.
Similarly, when an arbitrary total order is selected on $X$, there is exactly one order preserving isomorphism $\coprod_{x \in X} [f(x)] \times [m(x)] \cong [n]$.

We use these choices to induce `the' embeddings
\[ \Sigma_i \times \Sigma_j \subseteq \Sigma_n \text{, } \prod_{x \in X} \Sigma_{m(x)}^{\times f(x)} \subseteq \Sigma_n \text{, } G \wr \Sigma_i \times G \wr \Sigma_j \subseteq \Sigma_n \text{, and } \prod_{x \in X} (G \wr \Sigma_{m(x)})^{\times f(x)} \subseteq G \wr \Sigma_n, \]
when such a particular choice is desired.

\subsection{Transfer products} \label{subsec:tr prod}
Let $n=i+j$.
Given a finite $G \wr \Sigma_i$-set $X$ and a finite $G \wr \Sigma_j$-set $Y$, we will be interested in the induced $G \wr \Sigma_n$-set
\[ X \star Y = \Tr_\phi(X \boxtimes Y), \]
where $\phi$ is the embedding $G \wr \Sigma_i \times G \wr \Sigma_j \subseteq G \wr \Sigma_n$.
To avoid clutter, we suppress the groups from the notation and write $\Tr_{i,j}^n$ for the transfer along $\phi$.
This determines a functor
\[ \star \colon \Fin_{G \wr \Sigma_i} \times \Fin_{G \wr \Sigma_j} \to \Fin_{G \wr \Sigma_n} \]
and a homomorphism 
\[\star \colon A(G \wr \Sigma_i) \tensor A(G \wr \Sigma_j) \to A(G \wr \Sigma_n), \]
which we call the transfer product.

\subsection{The marks homomorphism} \label{subsec:marks}
Let $\Conj(G)$ denote the set of conjugacy classes of subgroups of $G$. 
The conjugacy class of a subgroup $H \subseteq G$ is denoted by $[H]$. There is a canonical bijection between the set $\Conj(G)$ and the set of isomorphism classes of transitive $G$-sets sending the conjugacy class $[H]$ to the isomorphism class $[G/H]$.

For a commutative ring $R$, let $\Marks(G,R)$ be the ring of $R$-valued functions on the set $\Conj(G)$. 
If $R$ is omitted, it is assumed to be $\Z$. This is often referred to as the ``ghost ring'' in the literature.
The ring $\Marks(G,R)$ has an $R$-linear basis given by the indicator functions 
\[ \1_{[H]} \colon [K] \mapsto \begin{cases} 1, & [H]=[K] \\ 0, & \text{otherwise} \end{cases} \]
for $[H] \in \Conj(G)$.

The marks homomorphism, which we also call the character map, is the ring homomorphism
\[
\chi = \chi_G \colon A(G) \to \Marks(G)
\]
that sends a $G$-set $X$ to the function on $\Conj(G)$ given by cardinalities of fixed point sets
\[
\chi([X]) \colon [H] \mapsto |X^H|.
\]
The marks homomorphism is injective and an isomorphism after rationalization. In other words,
\[ \Q \tensor \chi \colon \Q \tensor A(G) \to \Q \tensor \Marks(G) \cong \Marks(G,\Q) \]
is an isomorphism of commutative rings.

\subsection{Induced operations on marks} \label{subsec:marks ops}

Since $\Res_\phi$ and $\Tr_\phi$ are additive, rationalization induces maps
\[
\Res_\phi \colon \Marks(G,\Q) \to \Marks(H,\Q) \text{ and } \Tr_\phi \colon \Marks(H,\Q) \to \Marks(G,\Q)
\]
that fit into commutative squares
\[
\begin{tikzcd}
    A(G) \ar[r, "\Res_\phi"] \ar[d, "\chi_G"'] & A(H) \ar[d, "\chi_H"] \\ \Marks(G,\Q) \ar[r, "\Res_\phi"] & \Marks(H,\Q)
\end{tikzcd}
\text{ and }
\begin{tikzcd}
    A(H) \ar[r, "\Tr_\phi"] \ar[d, "\chi_H"'] & A(G) \ar[d, "\chi_G"] \\ \Marks(G,\Q) \ar[r, "\Tr_\phi"] & \Marks(H,\Q)
\end{tikzcd}
\]
witnessing compatibility with the marks homomorphism.

Similarly, since $\boxtimes$ and $\star$ are bilinear, rationalization induces maps
\[ \boxtimes \colon \Marks(G,\Q) \tensor \Marks(H,\Q) \to \Marks(G \times H,\Q) \]
and
\[ \star \colon \Marks(G \wr \Sigma_i,\Q) \tensor \Marks(G \wr \Sigma_j,\Q) \to \Marks(G \wr \Sigma_n,\Q) \]
fitting into commutative squares
\[
\begin{tikzcd}
    A(G) \tensor A(H) \ar[r, "\boxtimes"] \ar[d, "\chi_G \tensor \chi_H"'] & A(G \times H) \ar[d, "\chi_{G \times H}"] \\
    \Marks(G,\Q) \tensor \Marks(H,\Q) \ar[r, "\boxtimes"] & \Marks(G \times H,\Q)
\end{tikzcd}
\]
and
\[
\begin{tikzcd}
    A(G \wr \Sigma_i) \tensor A(G \wr \Sigma_j) \ar[r, "\star"] \ar[d, "\chi_{G\wr\Sigma_i} \tensor \chi_{G\wr\Sigma_j}"'] & A(G \wr \Sigma_n) \ar[d, "\chi_{G \wr \Sigma_n}"] \\
    \Marks(G\wr\Sigma_i,\Q) \tensor \Marks(G\wr\Sigma_j,\Q) \ar[r, "\star"] & \Marks(G\wr\Sigma_n,\Q)
\end{tikzcd}
\]
witnessing compatibility with the marks homomorphism.

\subsection{Formula for restriction on marks} \label{subsec:marks res}
To give a formula for the restriction along $\phi \colon H \to G$ on marks, consider the induced function
\[ \phi \colon \Conj(H) \to \Conj(G) \colon [K] \mapsto [\phi(K)]. \]
Precomposition with this function gives the restriction:
\[ \Res_\phi(f)([K]) = f(\phi[K]) = f([\phi(K)]) \]
for $f \in \Marks(G,\Q)$ and $[K] \in \Conj(H)$. 
In fact, this formula defines restrictions
\[ \Res_\phi \colon \Marks(G,R) \to \Marks(H,R) \]
for any coefficient ring $R$.

\subsection{Formula for transfer on marks} \label{subsec:marks tr}
The formula for the transfer map on marks is more complicated.
Along a subgroup inclusion $H \subseteq G$, the formula is given by
\[ \Tr_{H \subseteq G}(f)([K]) = \sum_{\substack{gH \in G/H \\ K^g \subseteq H}} f([K^g]) \]
for $f \in \Marks(H)$ and $[K] \in \Conj(G)$. 
This defines transfers
\[ \Tr_{H \subseteq G} \colon \Marks(H,R) \to \Marks(G,R) \]
for any coefficient ring $R$.
At the other extreme, when $\phi \colon G \twoheadrightarrow e$ is a surjection onto the trivial group, Burnside's orbit counting lemma gives the beautiful formula
\[
\Tr_{G \twoheadrightarrow e}(f)([e]) = \frac{1}{|G|}\sum_{g \in G} f([\ag g]),
\]
for $f \in \Marks(G,\Q)$, where $[\ag g]$ denotes the conjugacy class of the cyclic subgroup generated by $g \in G$.
Applied to $\chi([X])$ for a finite $G$-set $X$, the formula demonstrates that the number of orbits for the $G$-action on $X$ is equal to the average of the number of fixed points of $X$ by the elements of $G$.
More general transfer formulas can be found in \cite[Chapter 5]{Bouc2010}.

\subsection{Formula for external product on marks} \label{subsec:marks ext prod}
To give a formula for the external product on marks, note that for a finite $G$-set $X$, a finite $H$-set $Y$, and a subgroup $K \subseteq G \times H$,
\[ (X \times Y)^K \cong X^{\pi_1 K} \times Y^{\pi_2 K}, \]
where the $\pi_i$'s denote the projections.
From this, we conclude that 
\[ (f \boxtimes g)([K]) = f(\pi_1 [K])g(\pi_2 [K]) = f([\pi_1(K)])g([\pi_2(K)]) \]
for $f \in \Marks(G,\Q)$, $g \in \Marks(H,\Q)$, and $[K] \in \Conj(G \times H)$.
In fact, this formula defines external products
\[ \boxtimes \colon \Marks(G,R) \tensor \Marks(H,R) \to \Marks(G \times H,R) \] 
for any coefficient ring $R$.

\subsection{Formula for the transfer product on marks} \label{subsec:marks tr prod}

Let $n=i+j$.
The formula for the transfer product on marks is given by
\[ (f \star g)([H]) = \sum_{\substack{\sigma (\Sigma_i \times \Sigma_j) \in \Sigma_n / (\Sigma_i \times \Sigma_j) \\ H^\sigma \subseteq G \wr \Sigma_i \times G \wr \Sigma_j}} f(\pi_{[i]}[H^\sigma]) g(\pi_{[j]}[H^\sigma]) \]
for $f \in \Marks(G \wr \Sigma_i, \Q)$, $g \in \Marks(G \wr \Sigma_j, \Q)$, and $[H] \in \Conj(G \wr \Sigma_n)$ where $\pi_{[i]} \colon \Sigma_{[i]} \times \Sigma_{[j]} \to \Sigma_{[i]}$ denotes the projection and similarly for $\pi_{[j]}$.

In fact, since this formula is integral, it defines transfer products
\[ \star \colon \Marks(G \wr \Sigma_i,R) \tensor \Marks(G \wr \Sigma_j,R) \to \Marks(G \wr \Sigma_n,R) \]
for any coefficient ring $R$.

To see the correctness of the formula, notice that
\[
\smashoperator[r]{\sum_{\substack{\sigma (\Sigma_i \times \Sigma_j) \in \Sigma_n / (\Sigma_i \times \Sigma_j) \\ H^\sigma \subseteq G \wr \Sigma_i \times G \wr \Sigma_j}}} f(\pi_{[i]}[H^\sigma]) g(\pi_{[j]}[H^\sigma])
=
\smashoperator[r]{\sum_{\substack{(\bar g,\sigma) (G \wr \Sigma_i \times G \wr \Sigma_j) \in (G \wr \Sigma_n) / (G \wr \Sigma_i \times G \wr \Sigma_j) \\ H^{(\bar g,\sigma)} \subseteq G \wr \Sigma_i \times G \wr \Sigma_j}}} f(\pi_{[i]}[H^{(\bar g,\sigma)}]) g(\pi_{[j]}[H^{(\bar g,\sigma)}])
\]
using the bijection $\Sigma_n / (\Sigma_i \times \Sigma_j) \cong (G \wr \Sigma_n) / (G \wr \Sigma_i \times G \wr \Sigma_j)$ and that the latter expression is exactly the composition of the formulas for the external product and for $\Tr_{i,j}^n$.

\subsection{Commutative graded rings} \label{subsec:cgrings}

Define
\[ A(G \wr \Sigma) = \bigoplus_{n \geq 0} A(G \wr \Sigma_n) \text{ and } \Marks(G \wr \Sigma,R) = \bigoplus_{n \geq 0} \Marks(G \wr \Sigma_n, R). \]
When $R$ is omitted, it is assumed to be $\Z$.
The transfer product induces a commutative, associative, and unital operation on these graded abelian groups.
Thus, the transfer product gives these objects the structure of a commutative $\N$-graded ring (as opposed to a graded commutative ring, which satisfies $xy = (-1)^{|x||y|}yx$). 
The unit of $A(G \wr \Sigma)$ is $1 \in A(G \wr \Sigma_0) = A(e) \cong \Z$, and the unit of $\Marks(G \wr \Sigma,R)$ is the indicator function $1=\1_{[e]} \in \Marks(G \wr \Sigma_0,R) = \Marks(e,R) \cong R$ of $[e] \in \Conj(G \wr \Sigma_0)$.
The graded rings $A(G \wr \Sigma)$ and $\Marks(G \wr \Sigma)$ are connected in the sense that the degree 0 subrings $A(G \wr \Sigma_0)$ and $\Marks(G \wr \Sigma_0)$ are isomorphic to $\Z$.

The level-wise restriction and transfer maps along $\phi \colon H \to G$ are compatible with the transfer product and therefore induce graded ring maps\footnote{Note that this does not induce the structure of a Mackey functor on $G \mapsto A(G\wr\Sigma)$ as these restrictions and transfers do not obey the double coset formula, as is easily verifiable in degree 0.}
\[ \Res_\phi \colon A(G \wr \Sigma) \to A(H \wr \Sigma) \text{ and } \Tr_\phi \colon A(H \wr \Sigma) \to A(G \wr \Sigma). \]
There are also graded ring maps
\[ \Res_\phi \colon \Marks(G \wr \Sigma, R) \to \Marks(H \wr \Sigma, R) \text{ and } \Tr_\phi \colon \Marks(H \wr \Sigma, R) \to \Marks(G \wr \Sigma,R), \]
although the transfer is only defined when the coefficients of the linear map $\Tr_{\phi \wr \Sigma_n} \colon \Marks(H \wr \Sigma_n,\Q) \to \Marks(G \wr \Sigma_n,\Q)$ are elements of $R$.

Similarly, the level-wise marks homomorphisms induce a graded ring map
\[ \chi \colon A(G \wr \Sigma) \to \Marks(G \wr \Sigma), \]
which is an isomorphism after tensoring with $\Q$. 
In other words,
\[ \Q \tensor \chi \colon \Q \tensor A(G \wr \Sigma) \to \Q \tensor \Marks(G \wr \Sigma) \cong \Marks(G \wr \Sigma,\Q) \]
is an isomorphism.

\subsection{Graded monoid rings} \label{subsec:monoid rings}

An augmented monoid will refer to a monoid equipped with a homomorphism into $\N$ called the size or augmentation.
The augmentation of an augmented monoid $M$ is denoted by $M \to \N \colon m \mapsto \|m\|$.
Given an augmented monoid $M$, the monoid ring $\Z[M]$ is an $\N$-graded ring. 
This produces a functor from the category of augmented monoids to the category of $\N$-graded rings with grading preserving maps.

\subsection{Power operations} \label{subsec:pow ops}
For $n \geq 0$, the total power operation
\[
\P_n \colon A^+(G) \to A^+(G \wr \Sigma_n)
\]
is defined by $\P_n([X]) = [X^{\times n}]$ for $X$ a finite $G$-set.
These maps are multiplicative and satisfy an additivity relation in terms of the transfer product
\[
\P_n(X+Y) = \sum_{\substack{i+j = n \\ i, j \geq 0}} \P_i(X) \star \P_j(Y) = \sum_{\substack{i+j = n \\ i, j \geq 0}} \Tr_{i,j}^n(\P_i(X) \boxtimes \P_j(Y)) .
\]
The operation $\P_0$ is the constant function $1$, and $\P_n(0)=0$ for $n > 0$.
The total power operation extends to a map
\[
\P_n \colon A(G) \to A(G \wr \Sigma_n)
\]
by enforcing the above additivity relation.
Specifically, $\P_n(X-Y)$ may be computed inductively by using that
\[ \P_n(X) = \P_n( (X-Y) + Y ) = \P_n(X-Y) + \sum_{\substack{i+j=n \\ i \geq 0,\, j > 0}} \P_i(X-Y) \star \P_j(Y). \]
A closed form is given in \ref{cor:powofdiff}.

\subsection{The Frobenius--Wielandt map} \label{subsec:weird}
Let $G$ be a finite group, let $m = |G|$, and let $C_m$ be the cyclic group of order $m$. 
The Frobenius--Wielandt map is a strange ring map $A(C_m) \to A(G)$. 
This ring map is unusual since it is not induced by restriction along a group homomorphism.
Instead, it is induced by the map
\[
\Cyc \colon \Conj(G) \to  \Conj(C_m)
\]
sending a conjugacy class $[H]$ to the conjugacy class of the unique subgroup of $C_m$ of order $|H|$. 
This induces a ring map
\[
w = \Cyc^* \colon \Marks(C_m) \to \Marks(G).
\]
In \cite{DSY1992}, it is shown that the induced map $A(C_m) \to \Marks(G)$, given by precomposing $w$ with the marks homomorphism, factors through $A(G)$, providing the ring map $w \colon A(C_m) \to A(G)$.

\section{The summand $\AA(G,n) \subseteq A(G \wr \Sigma_n)$} \label{sec:submissives}

\begin{definition}
Let $\AA(G,n) \subseteq A(G \wr \Sigma_n)$ be the summand spanned by the smallest subset of the canonical basis whose span contains the image of the total power operation $\P_n \colon A(G) \to A(G \wr \Sigma_n)$. 
\end{definition}

In this paper, we give many equivalent descriptions of $\AA(G,n)$.
We begin with a description of $\AA(G,n)$ as a Grothendieck ring.

\begin{definition}
We say that a finite $G \wr \Sigma_n$-set $X$ is submissive if it admits a $G \wr \Sigma_n$-equivariant injection into a $G \wr \Sigma_n$-set of the form $Y^{\times n}$ for some $G$-set $Y$.
\end{definition}

The disjoint union and product of submissive $G \wr \Sigma_n$-sets are submissive:
If $X_1 \hookrightarrow {Y_1}^{\times n}$ and $X_2 \hookrightarrow {Y_2}^{\times n}$, then $X_1 \amalg X_2 \hookrightarrow (Y_1 \amalg Y_2)^{\times n}$ and $X_1 \times X_2 \hookrightarrow (Y_1 \times Y_2)^{\times n}$, and it is clear that the empty union $\emptyset$ and empty product $\ast$ are submissive $G \wr \Sigma_n$-sets.
Thus, we may consider the Grothendieck ring of (isomorphism classes of) submissive $G \wr \Sigma_n$-sets.

Any transitive $G \wr \Sigma_n$-subset of a submissive set is also submissive, so the category of submissive $G \wr \Sigma_n$-sets is generated under finite coproducts by the transitive submissive $G \wr \Sigma_n$-sets in the same way that the category of finite $G \wr \Sigma_n$-sets is generated under finite coproducts by the transitive $G \wr \Sigma_n$-sets.
Thus the Grothendieck ring of submissive $G \wr \Sigma_n$-sets is freely generated as an abelian group by the isomorphism classes of transitive submissive $G \wr \Sigma_n$-sets.
In this sense, we may consider the Grothendieck ring of submissive $G \wr \Sigma_n$-sets as a subring of $A(G \wr \Sigma_n)$.

\begin{proposition} \label{prop:AAisKsub}
The Grothendieck ring of submissive $G \wr \Sigma_n$-sets is the summand $\AA(G,n) \subseteq A(G \wr \Sigma_n)$.
\end{proposition}
\begin{proof}
First we show that the Grothendieck ring of submissive $G \wr \Sigma_n$-sets is contained in $\AA(G,n)$.
This follows from the fact that if $X \hookrightarrow Y^{\times n}$ is a submissive $G \wr \Sigma_n$-set, each of the transitive $G \wr \Sigma_n$-subsets of $X$ are also subsets of $Y^{\times n}$. 
Since $\P_n([Y]) = [Y^{\times n}]$ is in the image of the total power operation, each of these transitives is in $\AA(G,n)$, and thus $[X]$ is in $\AA(G,n)$. 

For the other inclusion, note that if $[(G \wr \Sigma_n)/K]$ appears in the basis decomposition of $\P_n([Y])$, then $(G \wr \Sigma_n)/K$ injects into $Y^{\times n}$ and so $(G \wr \Sigma_n)/K$ is a submissive $G \wr \Sigma_n$-set. 
Thus, we see that the Grothendieck ring of submissive $G \wr \Sigma_n$-sets contains the basis members required to express the image of the power operation $A^+(G) \to A(G \wr \Sigma_n)$.
Further, this is enough to express the image of the larger power operation $A(G) \to A(G \wr \Sigma_n)$ because each of the transitives appearing in the basis decomposition of $\P_n([X]-[Y])$ appear in the basis decomposition of $\P_n([X]+m[Y])$ for large enough $m$. 
Although the proof of this fact is too technical to appear here, we prove it as \ref{cor:pos enough} using the language of partitions and compositions to manipulate the relation describing power operations applied to sums.
Alternatively, one may use the graded ring structure of \ref{sec:alpha} together with \ref{cor:montoab} to complete the proof.
\end{proof}

\begin{proposition}
Let $\phi \colon H \to G$ be a group homomorphism. 
Then $\Res_{\phi \wr \Sigma_n} \colon A(G \wr \Sigma_n) \to A(H \wr\Sigma_n)$ restricts to a ring map
\[
\Res_{\phi} \colon \AA(G,n) \to \AA(H,n)
\]
and
$\Tr_{\phi \wr \Sigma_n} \colon A(H \wr \Sigma_n) \to A(G\wr \Sigma_n)$ restricts to a map of abelian groups
\[
\Tr_{\phi} \colon \AA(H, n) \to \AA(G,n).
\]
\end{proposition}
\begin{proof}
For the restriction map, it suffices to notice that the restriction of a submissive $G \wr \Sigma_n$-set along $\phi \wr \Sigma_n$ is submissive. This follows from the fact that if $Y$ is a finite $G$-set, then there is an isomorphism of $H \wr \Sigma_n$-sets $\Res_{\phi \wr \Sigma_n}(Y^{\times n}) \cong (\Res_{\phi}(Y))^{\times n}$.

For the transfer map, let $X$ be a submissive $H \wr \Sigma_n$-set and choose an injection $X \hookrightarrow Y^{\times n}$. The proposition follows from the fact that
\[
G \wr \Sigma_n \times_{H \wr \Sigma_n} X \hookrightarrow G \wr \Sigma_n \times_{H \wr \Sigma_n} Y^{\times n}
\]
is injective and there are isomorphisms of $G \wr \Sigma_n$-sets
\[
G \wr \Sigma_n \times_{H \wr \Sigma_n} Y^{\times n} \cong (G \times_H Y)^{\times n}. \qedhere
\]
\end{proof}

\begin{proposition}\label{prop:AAstar}
    The transfer product $\star \colon A(G \wr \Sigma_i) \otimes A(G \wr \Sigma_j) \to A(G \wr \Sigma_n)$ restricts to a transfer product
    \[ \star \colon \AA(G,i) \otimes \AA(G,j) \to \AA(G,n) \]
    where $n=i+j$.
\end{proposition}
\begin{proof}
    This follows from the fact that if $X_1 \hookrightarrow {Y_1}^{\times i}$ and $X_2 \hookrightarrow {Y_2}^{\times j}$ are submissive $G \wr \Sigma_i$ and $G \wr \Sigma_j$-sets respectively, then
    \[ X_1 \star X_2 \hookrightarrow {Y_1}^{\times i} \star {Y_2}^{\times j} \hookrightarrow \coprod_{i'+j'=n} {Y_1}^{\times i'} \star {Y_2}^{\times j'} \cong (Y_1 \amalg Y_2)^{\times n}. \qedhere \]
\end{proof}

\begin{remark}
We are certainly not the first to notice that the image of the total power operation is very small relative to the size of $A(G \wr \Sigma_n)$ and also reasonably well-behaved. For instance, this is noted in \cite[Remark 9.3]{SchwedeCourseNotes}.
\end{remark}

\section{Decorated partitions} \label{sec:decoratedpartitions}

Partitions, i.e.\ formal sums of positive integers, are important in the study of symmetric groups.
In this paper, we will make use of $X$-decorated partitions, where $X$ is a set. These are sometimes called colored partitions.
An $X$-decorated partition is a formal sum of pairs in $X \times \N_+$.
Such a pair is called a part.
A part is interpreted as its integer component, called the size of the part, `decorated' by the $X$ component.
The classical theory of integer partitions is recovered when $X=*$ is a singleton.

We write the collection of $X$-decorated partitions as $P(X) = \N\{X \times \N_+\}$.
We identify this collection with finite support functions
\[ \lambda \colon X \times \N_+ \to \N \colon x,m \mapsto \lambda_{x,m}. \]
The finite support function $\lambda$ is identified with the formal sum 
\[ \lambda = \sum_{x \in X,\, m \in \N_+} \lambda_{x,m} \cdot (x,m) \]
containing $\lambda_{x,m}$-many parts of size $m$ and decoration $x$.
When $\lambda \in P(*)$ is an integer partition, we write $\lambda_m$ instead of $\lambda_{*,m}$.
In this situation, it is important to note that $\lambda_m$ denotes the number parts of size $m$ in the partition and not the $m$th formal summand (in some chosen ordering).
Decorated partitions will be generally denoted by the lowercase Greek letters $\kappa,\lambda,\mu$ or less frequently $\gamma,\delta,\epsilon$.

We generalize the structure (and notation) of integer partitions to decorated partitions.
Formal summation gives $P(X)$ the structure of a free commutative monoid
\[ \kappa + \lambda \colon x,m \mapsto \kappa_{x,m} + \lambda_{x,m}. \]
The identity of this monoid is the empty partition 0.
The number of parts in a partition $\lambda$ is called its length and is denoted
\[ |\lambda| = \sum_{x,m} \lambda_{x,m}. \]
The sum of the sizes of the parts in a partition $\lambda$ is called its size and is denoted
\[ \|\lambda\| = \sum_{x,m} \lambda_{x,m} \cdot m. \]
A partition of size $n$ may also be called a partition of $n$. 
If $X$ is finite, there are finitely many $X$-decorated partitions of size $n$.
Length and size are additive so that $|\kappa + \lambda| = |\kappa| + |\lambda|$ and $\|\kappa + \lambda\| = \|\kappa\|+\|\lambda\|$.
The size map $P(X) \to \N$ gives $P(X)$ the structure of an augmented monoid.

We also define the factorial and multifactorial which are respectively given by
\[ \lambda! = \prod_{x,m} \lambda_{x,m}! \text{ and } \mf\lambda = \prod_{x,m} (m!)^{\lambda_{x,m}}. \]
The multifactorial is exponential so that $\mf{\kappa + \lambda} = \mf\kappa \cdot \mf\lambda$, but there is no such relationship for the factorial.
Given a function $f \colon X \to Y$, there is a pushforward map $f_* \colon P(X) \to P(Y)$ which corresponds to a fiber sum of finite support functions
\[ f_*\lambda \colon y,m \mapsto \sum_{x \in f^{-1}(y)} \lambda_{x,m}. \]
The pushforward $f_*$ preserves addition, length, size, and the multifactorial, and it preserves the factorial when $f$ is injective.

Of particular importance in this paper are partitions decorated by $\Conj(G)$. Recall that $\Conj(G)$ denotes the conjugacy classes of subgroups of $G$. 
\begin{definition}
Let $\Parts(G) = P(\Conj(G))$ denote the set of $\Conj(G)$-decorated partitions.
The set of $\Conj(G)$-decorated partitions of size $n$ is denoted by $\Parts(G,n)$.
\end{definition}
In this notation $\Parts(G) = \coprod_n \Parts(G,n)$.
We identify $\Parts(e)$ with the set of integer partitions $P(*)$.

\section{The graded rings $A(G \wr \Sigma)$ and $\AA(G)$} \label{sec:alpha}

Recall the commutative graded ring $A(G \wr \Sigma) = \bigoplus_{n \geq 0} A(G \wr \Sigma_n)$ from \ref{subsec:cgrings}.
\begin{definition}
Let $\Conj(G \wr \Sigma) = \coprod_{n \geq 0} \Conj(G \wr \Sigma_n)$ denote the collection of conjugacy classes of subgroups of $G \wr \Sigma_n$ as $n$ varies.
\end{definition}
Recall that there is a canonical bijection between the set $\Conj(G)$ and the set of isomorphism classes of transitive $G$-sets. This corresponds to the canonical additive basis for $A(G \wr \Sigma)$, where the conjugacy class $[H] \in \Conj(G \wr \Sigma_n)$ corresponds to the degree $n$ basis element $[(G \wr \Sigma_n)/H] \in A(G \wr \Sigma_n) \subseteq A(G \wr \Sigma)$.
\begin{proposition}
    The transfer product on $A(G \wr \Sigma)$ induces the structure of a commutative augmented monoid on $\Conj(G \wr \Sigma)$ via the embedding $\Conj(G \wr \Sigma) \hookrightarrow A(G \wr \Sigma)$.
    Explicitly, for $[H] \in \Conj(G \wr \Sigma_i)$ and $[K] \in \Conj(G \wr \Sigma_j)$,
    \[ [H] \star [K] = [H \times K] \in \Conj(G \wr \Sigma_n), \]
    where $n=i+j$ and $H \times K \subseteq G \wr \Sigma_i \times G \wr \Sigma_j \subseteq G \wr \Sigma_n$.
    The augmentation $\Conj(G \wr \Sigma) \to \N$ is given by $[H \subseteq G \wr \Sigma_n] \mapsto n$.
\end{proposition}
\begin{proof}
    The fact that this is associative, commutative, and unital follows from the fact that the transfer product has these properties and that the unit for the transfer product is transitive.
    It is also immediate that the augmentation is a homomorphism.
    The explicit formula follows from the fact that
    \begin{align*}
        [(G \wr \Sigma_i)/K] \star [(G \wr \Sigma_j)/J] &= \Tr_{i,j}^{n}([(G \wr \Sigma_i)/K] \boxtimes [(G \wr \Sigma_j)/J]) \\ &= \Tr_{i,j}^{n}([(G \wr \Sigma_i \times G \wr \Sigma_j)/(K \times J)]) \\
    &= [(G \wr \Sigma_n)/(K \times J)]. \qedhere
    \end{align*}
\end{proof}
\begin{corollary}
    The graded ring $A(G \wr \Sigma)$ is isomorphic to the graded monoid ring $\Z[\Conj(G \wr \Sigma)]$ where the grading is given by the augmentation as in \ref{subsec:monoid rings}.
\end{corollary}

We now consider an augmented monoid homomorphism $\Parts(G) \to \Conj(G \wr \Sigma)$.
This is easy to construct as $\Parts(G)$ is a free commutative monoid.
\begin{definition} \label{def:alpha}
    Let 
    \[ \alpha \colon \Parts(G) \to \Conj(G \wr \Sigma) \]
    be the augmented monoid homomorphism induced by sending the generator $([H],m) \in \Parts(G)$ to $[H \wr \Sigma_m] \in \Conj(G \wr \Sigma_m)$.
    For $\lambda \in \Parts(G,n)$, let $[(G \wr \Sigma_n)/\alpha(\lambda)]$ denote the transitive $G \wr \Sigma_n$-set corresponding to $\alpha(\lambda)$ under the canonical bijection between conjugacy classes of subgroups and isomorphism classes of transitive sets.
    Explicitly, $[(G \wr \Sigma_n)/\alpha(\lambda)] = [(G \wr \Sigma_n)/H]$ for any representative $H \in \alpha(\lambda)$ of the conjugacy class $\alpha(\lambda)$.
\end{definition}
By construction, $\alpha$ is augmented i.e.\ $\alpha(\lambda) \in \Conj(G \wr \Sigma_{\|\lambda\|})$ for each $\lambda$.
Accordingly, $\alpha$ restricts to maps $\Parts(G,n) \to \Conj(G \wr \Sigma_n)$.
The homomorphism $\alpha$ has the explicit formula
\[ \alpha(\lambda) = \bigstar_{[H],m} [H \wr \Sigma_m]^{\star \lambda_{[H],m}} = \left[ \prod_{[H],m} (H \wr \Sigma_m)^{\times \lambda_{[H],m}} \right] \]
for $\lambda \in \Parts(G)$, where $\prod_{[H],m} (H \wr \Sigma_m)^{\times \lambda_{[H],m}} \subseteq \prod_{[H],m} (G \wr \Sigma_m)^{\times \lambda_{[H],m}} \subseteq G \wr \Sigma_{\|\lambda\|}$ as in \ref{subsec:fin sets}.

Notice that $\alpha$ is injective.
This can be proved directly from the definition, but it turns out that $\alpha$ admits a canonical retract. 
This retract is described in \ref{sec:beta}. 
\begin{definition}
    Let $\AA(G) = \bigoplus_{n \geq 0} \AA(G,n)$.
    By \ref{prop:AAstar}, this is a commutative graded subring of $A(G \wr \Sigma)$.
\end{definition}

\begin{proposition}\label{prop:transitive power}
The total power operation takes the isomorphism class of a transitive $G$-set to the isomorphism class of a transitive $G \wr \Sigma_n$-set.
In particular,
\[ \P_n([G/H]) = [(G/H)^{\times n}] = [(G\wr \Sigma_n)/(H \wr \Sigma_n)]. \]
\end{proposition}
\begin{proof}
This follows from the isomorphism of $G \wr \Sigma_n$-sets
\[ (G/H)^{\times n} \cong (G\wr \Sigma_n)/(H \wr \Sigma_n). \qedhere \]
\end{proof}

\begin{proposition} \label{prop:AAMonRing}
The graded subring $\AA(G)$ is the image of the injective map of graded rings $\Z[\Parts(G)] \xhookrightarrow{\alpha} \Z[\Conj(G \wr \Sigma)] \cong A(G \wr \Sigma)$, where the degree of $\lambda \in \Parts(G)$ is $\|\lambda\|$.
In particular, $\AA(G)$ is canonically isomorphic to the graded monoid ring $\Z[\Parts(G)]$.
\end{proposition}
\begin{proof}
    Keeping in mind the isomorphism of \ref{prop:transitive power}, we see that a generator $\lambda \in \Parts(G,n)$ of $\Z[\Parts(G)]$ corresponds to the isomorphism class of transitive $G \wr \Sigma_n$-sets
    \begin{equation} \label{eq:alpha star prod}
        [(G \wr \Sigma_n)/\alpha(\lambda)] = \bigstar_{[H],m} [(G \wr \Sigma_m)/(H \wr \Sigma_m)]^{\star \lambda_{[H],m}} = \bigstar_{[H],m} \P_m([G/H])^{\star \lambda_{[H],m}}.
    \end{equation}
    Note that the formula for the total power operation on an arbitrary sum is
    \begin{equation} \label{eq:power op on a sum}
        \P_n\left(\sum_{k=1}^{m} X_k\right) = \sum_{\substack{i_1+\ldots+i_m = n\\i_1,\ldots,i_m \geq 0}}  \bigstar_{k = 1}^{m} \P_{i_k}(X_k).
    \end{equation}
    It follows that $[(G \wr \Sigma_n)/\alpha(\lambda)]$ appears as a summand in
    \[ \P_n\left(\sum_{[H],m} \lambda_{[H],m} \cdot [G/H] \right) \]
    and hence is the isomorphism class of a submissive $G \wr \Sigma_n$-set.
    By \ref{prop:AAisKsub}, this demonstrates that the image of $\Z[\Parts(G)] \hookrightarrow A(G \wr \Sigma)$ is contained in $\AA(G)$.

    For the reverse inclusion, note that if each of the $X_k$'s in \eqref{eq:power op on a sum} is taken to be an isomorphism class of a transitive $G$-set, then each of the summands in the formula is of the form $[(G \wr \Sigma_n)/\alpha(\lambda)]$ for some $\lambda$.
    Thus, $\Z[\Parts(G)] \hookrightarrow A(G \wr \Sigma)$ contains all the basis elements necessary for it to contain the image of $\P_n \colon A^+(G) \to A(G \wr \Sigma_n)$ and thus the image of $\P_n \colon A(G) \to A(G \wr \Sigma_n)$ by \ref{cor:pos enough}.
    This demonstrates that the image of $\Z[\Parts(G)] \hookrightarrow A(G \wr \Sigma)$ contains $\AA(G)$, which completes the proof.
\end{proof}

\begin{corollary}
    The commutative graded ring $\AA(G)$ is canonically isomorphic to the graded monoid ring $\Z[\Parts(G)]$.
\end{corollary}
\begin{corollary}
    The subgroup $\AA(G,n) \subseteq A(G \wr \Sigma_n)$ is the subgroup generated by the isomorphism classes of transitive $G \wr \Sigma_n$-sets of the form $[(G \wr \Sigma_n)/\alpha(\lambda)]$ for $\lambda \in \Parts(G,n)$.
\end{corollary}
\begin{corollary} \label{cor:AAisPoly}
    The commutative graded ring $\AA(G)$ is an $\N$-graded polynomial ring with generators $[(G/H)^{\times n}] = \P_n([G/H])$ in degree $n$.
\end{corollary}
\begin{proof}
    This is immediate from the fact that a monoid ring on a free commutative monoid is a polynomial ring and from the correspondence between $([H],n) \in \Parts(G,n)$  and $[(G/H)^{\times n}] \in \AA(G,n)$.
\end{proof}

\begin{corollary}
    The ring $\AA(G)$ is the smallest subring of $A(G \wr \Sigma)$ containing $\P_n(A(G)) \subseteq A(G \wr \Sigma_n)$ for all $n$. 
\end{corollary}
\begin{proof}
    Use \ref{cor:AAisPoly} together with the fact that containing the image of $\P_n \colon A^+(G) \to A(G \wr \Sigma_n)$ is enough to contain the image of $\P_n \colon A(G) \to A(G \wr \Sigma_n)$ as in \ref{cor:pos enough}.
\end{proof}
\begin{corollary}
    The summand $\AA(G,n)$ is the smallest subgroup of $A(G \wr \Sigma_n)$ containing the image of the power operation $\P_n \colon A(G) \to A(G \wr \Sigma_n)$ and the image of the transfer product $\star \colon \AA(G,i) \tensor \AA(G,j) \to A(G \wr \Sigma_n)$ for all $i+j=n$ with $i,j>0$.
\end{corollary}

\section{A universal family of power operations} \label{sec:universalpowerops}
In this section, we provide a description of $\AA(G) = \bigoplus_{n \geq 0}\AA(G,n)$ as an object carrying a universal family of power operations out of $A(G)$. 

\begin{definition}
Let $(A,+)$ be an abelian group and $R = \bigoplus_{n \geq 0} R_n$ be a commutative graded ring.
An $R$-valued family of power operations on $A$ will refer to a collection of functions $P_n \colon A \to R_n$ for $n\geq 0$ such that $P_0(x)=1$ for all $x \in A$, $P_n(0)=0$ for $n>0$, and 
\[ 
\displaystyle P_n(x+y) = \sum_{i+j=n} P_i(x)P_j(y), 
\]
for all $x,y \in A$.
We denote the collection of $R$-valued families of power operations on $A$ as $\Pops(A,R)$. 
\end{definition}

This construction assembles into a functor
\[ \Pops \colon \Ab^\op \times \CGrRing \to \Set, \]
where $\CGrRing$ denotes the category of commutative graded rings (maps preserve the grading).

\begin{proposition} \label{prop:powcorep}
    The functor $\Pops$ is corepresentable in the first variable.
    That is, there is a functor 
    \[ \Pow\colon \Ab \to \CGrRing \]
    with a natural bijection
    \[ \Pops(A,R) \cong \CGrRing(\Pow(A),R). \]
\end{proposition}
\begin{proof}
    This is best seen in terms of generators and relations.
    Let $\Pow(A)$ be the commutative graded ring generated in degree $n$ by elements of the form $P_n(a)$ for each $a\in A$ subject to the relations $P_0(a)=1$, $P_n(0)=0$ for $n>0$, and 
    \[ 
    \displaystyle P_n(x+y) = \sum_{i+j=n} P_i(x)P_j(y). 
    \]
    By construction, we have a natural isomorphism $\Pops(A,R) \cong \CGrRing(\Pow(A),R)$.
\end{proof}
\begin{corollary} \label{cor:surjpow}
    If $f \colon A \to B$ is a surjective homomorphism of abelian groups, the induced map $\Pow(f)\colon \Pow(A) \to \Pow(B)$ is surjective.
\end{corollary}
\begin{proof}
    The map $\Pow(f)$ surjects onto the generators of $\Pow(B)$ in the above description.
\end{proof}

Let $R$ be a commutative graded ring and let $\widehat{R} = \prod_{n\geq 0} R_n$ denote the completion of $R$ at the system of ideals $R_{\geq n}$ for $n \in \N$. 
We refer to $\widehat R$ as the formal completion of $R$.
The formal completion $\widehat R$ inherits a ring structure from $R$.
The addition is component wise, and the multiplication is given by
\[ (rs)_n = \sum_{\substack{i+j=n \\ i,j \geq 0}} r_i s_j. \]
This is motivated by identifying $r \in \widehat R$ with the infinite sum $\sum_{n \geq 0} r_n$.
The units of the formal completion are exactly those $r \in \widehat R$ such that $r_0$ is a unit in $R_0$.
The corresponding inversion formula for elements of this form is identical to the inversion formula for formal power series.
A unit is called strict when $r_0=1$.
Let $\widehat{R}^s$ denote the group of strict units i.e.\ the group
\[ \widehat{R}^s = \left\{ r \in \widehat{R} \mid r_0 = 1 \right\}. \]

\begin{proposition} \label{prop:powrep}
    The functor $\Pops(A,R)$ is representable in the second variable by the group of strict units in the formal completion of $R$. 
    That is, there is a natural bijection
    \[ \Pops(A,R) \cong \Ab(A, \widehat{R}^s). \]
\end{proposition}
\begin{proof}
    For some $R$-valued family of power operations on $A$ and each $a \in A$, the sequence $(P_n(a))_{n \geq 0} \in \widehat R$ is a strict unit.
    The axioms of the power operations exactly correspond to this being a homomorphism $A \to \widehat{R}^s$.
    Thus, $\Pops(A,R) \cong \Ab(A,\widehat{R}^s)$.
\end{proof}
\begin{corollary} \label{cor:montoab}
    Any graded subring of $A(G \wr \Sigma)$ containing the image of $\P_n \colon A^+(G) \to A(G \wr \Sigma_n)$ for all $n$ automatically contains the image of image of $\P_n \colon A(G) \to A(G \wr \Sigma_n)$ for all $n$.
\end{corollary}
\begin{proof}
    There is a monoid homomorphism
    \[ A^+(G) \to \widehat{A(G \wr \Sigma)}^s \colon [X] \mapsto (\P_n([X])_{n \geq 0} \]
    and any monoid homomorphism from $A^+(G)$ to a group extends to a group homomorphism from $A(G)$.
\end{proof}

Recall from \ref{sec:decoratedpartitions} that $P(X)$ denotes the set of partitions decorated by $X$, which becomes a free commutative monoid under formal addition of decorated partitions.

\begin{corollary} \label{cor:basis}
    Let $X$ be a set and let $\Z\{X\}$ be the free abelian group on $X$.
    Then there are canonical isomorphisms of commutative graded rings
    \[ \Pow(\Z\{X\}) \cong \Z[(x,n) \mid x \in X,\, n>0,\, \deg (x,n)=n] \cong \Z[P(X)]. \]
\end{corollary}
\begin{proof}
    The second isomorphism, between the polynomial ring and the monoid ring, follows from the definition of $P(X)$ as a free commutative monoid.
    For the first isomorphism, we see that $\Pow(\Z\{X\})$ and the polynomial ring $\Z[(x,n)]$ have the same universal property since they both represent functions from $X$ into the group of strict units.
    Explicitly, a map $f \colon \Z[(x,n)] \to R$ is uniquely determined by elements $f(x,n) \in R_n$ for $x \in X$ and $n>0$, which is the data of a function $X \to \widehat{R}^s$. 
\end{proof}

\begin{corollary} \label{cor:powang}
    There is a canonical isomorphism of commutative graded rings
    \[ \Pow(A(G)) \cong \AA(G). \]
\end{corollary}
\begin{proof}
    \ref{prop:AAMonRing} shows that $\AA(G)$ is canonically isomorphic to the graded monoid ring $\Z[\Parts(G)]$. The corollary then follows immediately from \ref{cor:basis} with $X = \Conj(G)$.
\end{proof}

\begin{definition} \label{def:Fang}
    Given an abelian group homomorphism $F \colon A(H) \to A(G)$, let
    \[ \r F\colon \AA(H) \to \AA(G) \]
    denote the map induced by \ref{cor:powang} and $\Pow(F)$.
\end{definition}

\section{Burnside rings and representation rings} \label{sec:RU}

Let $RU(G)$ be the complex representation ring of $G$ and let $RU(G \wr \Sigma) = \bigoplus_{n \geq 0} RU(G \wr \Sigma_n)$. This object was thoroughly studied by Zelevinsky in \cite{Zelevinsky}. 

There is a commutative graded ring structure on $RU(G \wr\Sigma)$ by making use of a transfer multiplication defined similarly to $\star$. 
The $RU(G\wr \Sigma)$-valued family of power operations $(P_n \colon RU(G) \to RU(G \wr \Sigma_n))_{n \geq 0}$ given by $P_n([V]) = [V^{\otimes n}]$ induces a map of commutative graded rings $\Pow(RU(G)) \to RU(G \wr \Sigma)$.
We will show that this map is an isomorphism.
There are at least two ways to do this.
One is to prove results similar to \ref{sec:alpha}, replacing the set of isomorphism classes of transitive $G$-sets with the set of isomorphism classes of irreducible $G$-representations.
Another is to make use of the ``transfer ideal.''
For completeness, we will take the second approach because it also provides another way to prove \ref{cor:powang}.
This result follows easily from the work of Zelevinsky, but we include a complete proof with no claim to originality.

Let $(R, \star)$ be a commutative $\N$-graded ring. Let $I_n \subseteq R_n$ denote the image of the sum of the multiplication maps from lower degrees
\[
\star \colon \bigoplus_{\substack{i+j=n \\ i,j>0}} R_i \otimes R_j \to R_n.
\]
Equivalently, let $I=(R_{>0})^{\star 2}$ denote the square of the irrelevant ideal $R_{>0} = \bigoplus_{n>0} R_n \subseteq R$ and let $I_n$ be the $n$-th graded piece of $I$.

\begin{proposition}
The canonical map of commutative graded rings
\[
\Pow(RU(G)) \to RU(G \wr \Sigma)
\]
is an isomorphism.
\end{proposition}
\begin{proof}
We prove this by induction on the degree. 

We have a map of exact sequences
\[
\begin{tikzcd}
    \displaystyle\bigoplus_{\substack{i+j=n \\ i,j>0}} \Pow(RU(G))_i \otimes \Pow(RU(G))_j \ar[r] \ar[d, "\cong"', shorten <= -1.8em, pos = 0] & \Pow(RU(G))_n \ar[d] \ar[r] & \Pow(RU(G))_n/I_n \ar[d, "\cong"] \ar[r] & 0 \ar[d] \\ \displaystyle\bigoplus_{\substack{i+j=n \\ i,j>0}} RU(G\wr \Sigma_i) \otimes RU(G\wr \Sigma_j) \ar[r] & RU(G\wr \Sigma_n) \ar[r] & RU(G\wr \Sigma_n)/I_n \ar[r] & 0.
\end{tikzcd}
\]
By induction on $n$, the left vertical map is an isomorphism.
For the right vertical map, \ref{cor:basis} implies that $\Pow(RU(G))_n/I_n \cong RU(G)$ and \cite[Lemma 1.7 and Section 7.2]{Zelevinsky} implies that $RU(G \wr \Sigma_n)/I_n \cong RU(G)$.
By construction, the right map sends $P_n(V)$ to $[V^{\otimes n}]$, which is an isomorphism.
This implies that the middle map is surjective, but since both abelian groups are free of the same rank, the middle map is an isomorphism.
\end{proof}

At this stage, we can give a generalization of the following well-known result \cite[Section 7.3]{Fulton}. Consider the composite of ring maps
\[
\AA(e,n) \to A(\Sigma_n) \to RU(\Sigma_n),
\]
where $RU(\Sigma_n)$ is the complex representation ring of $\Sigma_n$ and $A(\Sigma_n) \to RU(\Sigma_n)$ is the linearization map taking a finite $\Sigma_n$-set to the associated permutation representation.
This map turns out to be an isomorphism.
That is, the linearization map $A(\Sigma_n) \to RU(\Sigma_n)$ splits as a map of commutative rings.

The ring $\AA(e,n)$ has been studied by others.
In particular, in \cite{partialyoung}, Oda, Takegahara, and Yoshida refer to it as the partial Burnside ring relative to the Young subgroups of the symmetric group.
In this sense, the subgroups $H \subseteq G \wr \Sigma_n$ such that $(G \wr \Sigma_n)/H$ is submissive are a generalization of the notion of a Young subgroup of $\Sigma_n$ to $G \wr \Sigma_n$. 

\begin{corollary}
Assume the linearization map $A(G) \to RU(G)$ is surjective (resp. an isomorphism), then the induced map $\AA(G,n) \to RU(G \wr \Sigma_n)$ is surjective (resp. an isomorphism). In particular, the composites
\[
\AA(e, n) \to A(\Sigma_n) \to RU(\Sigma_n)
\]
and
\[
\AA(\Sigma_2, n) \to A(\Sigma_2 \wr \Sigma_n) \to RU(\Sigma_2 \wr \Sigma_n)
\]
are isomorphisms.
\end{corollary}
\begin{proof}
This follows immediately from \ref{cor:surjpow} together with the fact that the linearization maps $A(e) \to RU(e)$ and $A(\Sigma_2) \to RU(\Sigma_2)$ are isomorphisms.
\end{proof}

\section{$\Parks(G)$} \label{sec:parks}

We seek to produce a kind of character map for $\AA(G)$ analogous to the character map $A(G) \to \Marks(G)$.
In this section we construct and study the target ring.

\begin{definition} \label{def:parks}
    For a finite group $G$ and a commutative ring $R$, let $\Parks(G,R)$ denote the abelian group of finite support functions $\Parts(G) \to R$.
    We give $\Parks(G,R)$ an $\N$-grading by declaring the homogeneous elements of degree $n$ to be those functions supported on partitions of size $n$.
    We identify the graded pieces with the abelian groups $\Parks(G,n,R) = \Fun(\Parts(G,n),R)$ so that $\Parks(G,R) \cong \bigoplus_n \Parks(G,n,R)$.
    When $R$ is omitted, it is assumed to be $\Z$.
\end{definition}

The graded ring $\Parks(G,R)$ has an $R$-linear basis given by the indicator functions
\[ \1_\lambda \colon \kappa \mapsto \begin{cases} 1, & \lambda=\kappa \\ 0, & \text{otherwise} \end{cases} \]
of homogeneous degree $\|\lambda\|$ for $\lambda \in \Parts(G)$.

\begin{proposition} \label{prop:parksprod}
    Define a graded multiplication $\star \colon \Parks(G,R) \otimes \Parks(G,R) \to \Parks(G,R)$ given by
    \[ (f \star g)(\mu) = \sum_{\kappa+\lambda=\mu} \frac{\mu!}{\kappa!\lambda!} f(\kappa)g(\lambda) \]
    for $\mu \in \Parts(G)$.
    This gives $\Parks(G,R)$ the structure of a commutative graded ring.
\end{proposition}
\begin{proof}
    The commutativity and bilinearity of $\star$ are clear.
    That $\star$ is graded follows from the fact that the size of partitions is additive.
    The $\star$-unit is the indicator function of the empty partition $0 \in \Parts(G,0)$.
    To see that $\star$ is associative, note that
    \begin{align*}
        (f \star (g \star h))(\lambda) 
        &= \sum_{\gamma + \kappa = \lambda} \frac{\lambda!}{\gamma!\kappa!} f(\gamma)(g \star h)(\kappa)
        \\&= \sum_{\gamma + \kappa = \lambda} \frac{\lambda!}{\gamma!\kappa!} f(\gamma) \sum_{\delta+\epsilon = \kappa} \frac{\kappa!}{\delta!\epsilon!} g(\delta)h(\epsilon)
        \\&= \sum_{\gamma + \delta + \epsilon = \lambda} \frac{\lambda!}{\gamma!\delta!\epsilon!} f(\gamma) g(\delta)h(\epsilon)
    \end{align*}
    and that $((f \star g) \star h)$ expands to this same expression.
\end{proof}
Note that the graded pieces $\Parks(G,n,R)$ are commutative rings, but this structure is not induced by the $\star$-product.

Next we produce a $\Parks(G,R)$-valued family of power operations 
\[
(\P_n \colon \Marks(G) \to \Parks(G,n))_{n \geq 0}. 
\]

\begin{proposition} \label{prop:powopformula}
The maps $\P_n\colon \Marks(G) \to \Parks(G,n)$ given by 
\[ \P_n(f)(\lambda) = \prod_{[H],m} f([H])^{\lambda_{[H],m}} \]
are a family of power operations.
\end{proposition}
\begin{proof}
It is clear that $\P_0=1$ and that $\P_n(0)=0$ for $n>0$.
Thus, it suffices to check that
\begin{align*}
    \P_n(f+g)(\mu) 
    &= \prod_{[H],m} ((f+g)([H]))^{\mu_{[H],m}}
    \\&= \prod_{[H],m} (f([H])+g([H]))^{\mu_{[H],m}}
    \\&= \prod_{[H],m} \left( \sum_{\kappa_{[H],m} + \lambda_{[H],m} = \mu_{[H],m}} \frac{\mu_{[H],m}!}{\kappa_{[H],m}!\lambda_{[H],m}!} f([H])^{\kappa_{[H],m}}g([H])^{\lambda_{[H],m}} \right)
    \\&= \sum_{\kappa+\lambda=\mu} \frac{\mu!}{\kappa!\lambda!} \P_{\|\kappa\|}(f)(\kappa) \P_{\|\lambda\|}(g)(\lambda)
    \\&= \sum_{i+j=n} \sum_{\substack{\kappa+\lambda = \mu \\ \|\kappa\| = i,\, \|\lambda\|=j}} \frac{\mu!}{\kappa!\lambda!} \P_i(f)(\kappa) \P_j(g)(\lambda)
    \\&= \sum_{i+j=n} \P_i(f)\star \P_j(g). \qedhere
\end{align*}
\end{proof}

\begin{corollary} \label{cor:trivial powopformula}
    The maps $\P_n \colon A(G) \to \Parks(G,n)$ given by
    \[ \P_n(X)(\lambda) = \prod_{[H],m} \chi(X)([H])^{\lambda_{[H],m}} \]
    are a family of power operations.
\end{corollary}
\begin{proof}
    This follows from the fact that this formula is the formula of \ref{prop:powopformula} precomposed with the marks homomorphism.
\end{proof}

We can use this family of power operations along with \ref{prop:powcorep} to produce a map $\AA(G) \to \Parks(G)$.
\begin{definition} \label{def:aachar}
    Let
    \[ \chi = \chi_G \colon \AA(G) \to \Parks(G) \]
    be the map obtained by composing the isomorphism $\AA(G) \cong \Pow(A(G))$ of \ref{cor:powang} and the map $\Pow(A(G)) \to \Parks(G)$ induced by \ref{cor:trivial powopformula} and \ref{prop:powcorep}.
    We call this the character map or parks homomorphism in analogy to the terminology for the map $\chi \colon A(G) \to \Marks(G)$.
\end{definition}
Alternatively, $\chi \colon \AA(G) \to \Parks(G)$ is the composite of the isomorphism $\AA(G) \cong \Pow(A(G))$, $\Pow(\chi)\colon \Pow(A(G)) \to \Pow(\Marks(G))$, and the map $\Pow(\Marks(G)) \to \Parks(G)$ induced by \ref{prop:powopformula} and \ref{prop:powcorep}.

On the polynomial generators of \ref{cor:AAisPoly}, the parks homomorphism is given by
\[ \chi \colon \P_n([G/H]) \in \AA(G,n) \mapsto \left( \lambda \mapsto \prod_{[K],m} \chi([G/H])([K])^{\lambda_{[K],m}} \right) \in \Parks(G,n) \]
for a subgroup $H \subseteq G$ and $n \geq 0$.

Although the parks homomorphism and the marks homomorphism $\AA(G) \to A(G \wr \Sigma) \to \Marks(G \wr \Sigma)$ are written with the same symbol, it is usually easy to distinguish them in context; the marks homomorphism is evaluated at conjugacy classes of subgroups while the parks homomorphism is evaluated at decorated partitions.

The parks homomorphism $\chi \colon \AA(G) \to \Parks(G)$ will be studied more thoroughly in \ref{sec:pullback}.
In particular, we will show it is closely related to the marks homomorphism $\chi \colon A(G \wr \Sigma) \to \Marks(G \wr \Sigma)$.

\section{The retract $\beta$} \label{sec:beta}

Recall the augmented monoid homomorphism $\alpha \colon \Parts(G) \to \Conj(G \wr \Sigma)$ as described in \ref{sec:alpha}.
In this section, we produce an augmented monoid retract
\[ \beta \ \colon \Conj(G \wr \Sigma) \to \Parts(G) \]
of $\alpha$.

Consider a subgroup $H \subseteq G \wr \Sigma_n$.
Choose a set $[n]$ of size $n$ and identify $\Sigma_n$ with the group $\Sigma_{[n]}=\Aut([n])$ of bijections of $[n]$ as in \ref{subsec:fin sets}.
For $i,j \in [n]$, define
\[ H_{ij} = \{ \bar g_j \mid (\bar{g},\sigma) \in H,\, \sigma\colon i \mapsto j \} \subseteq G. \]
Notice that these sets are subsets of $G$ that satisfy compositional relations similar to that of a groupoid:
    \[ e \in H_{ii} \text{, } H_{jk}H_{ij} \subseteq H_{ik} \text{, and } H_{ij}^{-1} = H_{ji}. \]
To see the middle property, choose $(\bar{g},\sigma),(\bar{h},\tau) \in H$ such that $\sigma \colon i \mapsto j$ and $\tau \colon j \mapsto k$.
Then, the multiplication
\[ (\bar{h},\tau)(\bar{g},\sigma) = (\bar{h}\tau \bar{g},\tau\sigma) \]
gives us
\[ (\bar{h}\tau \bar{g})_k = \bar h_k \bar g_{\tau^{-1}(k)} = \bar h_k \bar g_j. \]
so that $H_{jk}H_{ij} \subseteq H_{ik}$.
The proofs of the other two properties are similar.
It follows that $H_{ii}$ is a subgroup of $G$.

Notice that $H$ has a left action on $[n]$ via the quotient $G \wr \Sigma_n \twoheadrightarrow \Sigma_n$. 
Accordingly, we write $H \under [n]$ (read $H$ under $[n]$) for the collection of orbits of $[n]$ under the left action of $H$. 
Elements of $H \under [n]$ are of the form $Hi$ for $i \in [n]$.
The set $H_{ij}$ is nonempty if and only if $Hi=Hj$.

\begin{definition} \label{def:beta}
    For a subgroup $H \subseteq G \wr \Sigma_n$, choose $S \subseteq [n]$ containing a unique representative of each orbit in $H \under [n]$ (equivalently choose a section of the quotient map $[n] \to H \under [n]$).
    Then, define
    \[ \beta([H]) = \sum_{i \in S} ([H_{ii}], |Hi|). \]
    Once we show that this sum is independent of the choice of $S$, we may write
    \[ \beta([H]) = \sum_{Hi \in H\under [n]} ([H_{ii}], |Hi|) \in \Parts(G) \]
    without reference to the choice $S$.
\end{definition}
\begin{proposition} \label{prop:beta}
    The formula above is a well-defined homomorphism
    \[ \beta \colon \Conj(G \wr \Sigma) \to \Parts(G) \]
    of augmented monoids, and $\beta$ is a retract of $\alpha$.
\end{proposition}
\begin{proof}
We begin by showing this definition is independent of the choice of $S$.
To see this, note that if $i$ and $j$ are in the same $H$-orbit, then we may choose $h \in H_{ij}$ so that $hH_{ii}h^{-1} = H_{jj}$.

The fact that this definition is independent of the chosen representative $H$ of the conjugacy class $[H]$ follows from the behavior of the sets $H_{ij}$ under conjugation by elements of $H \subseteq G \wr \Sigma_n$.
First, note the following formula for conjugation in the wreath product:
\[ (\bar h, \tau)^{(\bar g, \sigma)} = (\bar g, \sigma)^{-1}(\bar h, \tau)(\bar g, \sigma) = (\sigma^{-1} (\bar g\,^{-1} \bar h \tau \bar g), \sigma^{-1}\tau \sigma). \]
If we write 
\[ \bar h^{(\bar g, \sigma)} \text{ for } \sigma^{-1} (\bar g\,^{-1} \bar h \tau \bar g) \]
then the conjugation formula is much simpler:
\[ (\bar h, \tau)^{(\bar g, \sigma)} = (\bar h^{(\bar g, \sigma)}, \tau^\sigma). \]
This is justified by identifying $G^{\times n}$ with the normal subgroup $G^{\times n} \trianglelefteq  G \wr \Sigma_n$ and considering the conjugation action of $G \wr \Sigma_n$ on this normal subgroup.
Additionally, note that
\[ \bar h^{(\bar g, \sigma)}{}_i = {\bar g_{\sigma(i)}}^{-1} \bar h_{\sigma(i)} \bar g_{\tau^{-1}(\sigma(i))} \]
for $i \in [n]$.

With this in mind, consider $(\bar g, \sigma) \in G \wr \Sigma_n$ and compute that
\begin{align*}
    {H^{(\bar g,\sigma)}}_{ij} 
    &= \{ \bar h_j \mid (\bar h, \tau) \in H^{(\bar g,\sigma)}, \tau \colon i \mapsto j \}
    \\&= \{ \bar h_j \mid (\bar h, \tau)^{(\bar g,\sigma)^{-1}} \in H, \tau \colon i \mapsto j \}
    \\&= \{ \bar h^{(\bar g, \sigma)}{}_j \mid (\bar h, \tau) \in H, \tau^\sigma \colon i \mapsto j \}
    \\&= \{ {\bar g_{\sigma(j)}}^{-1} \bar h_{\sigma(j)} \bar g_{\tau^{-1}(\sigma(j))} \mid (\bar h, \tau) \in H, \tau \colon \sigma(i) \mapsto \sigma(j) \}
    \\&= {\bar g_{\sigma(j)}}^{-1} \{ \bar h_{\sigma(j)} \mid (\bar h, \tau) \in H, \tau \colon \sigma(i) \mapsto \sigma(j) \} \bar g_{\sigma(i)}
    \\&= {\bar g_{\sigma(j)}}^{-1} H_{\sigma(i)\sigma(j)} \bar g_{\sigma(i)}.
\end{align*}
Now, the fact that the definition of $\beta$ is independent of the chosen representative $H$ of $[H]$ follows from the fact that
\[ {H^{(\bar{g},\sigma)}}_{ii} = \bar g_{\sigma(i)}^{-1} H_{\sigma(i)\sigma(i)} \bar g_{\sigma(i)} \]
and that if $S$ is a set of orbit representatives for $H \under [n]$, then $\sigma^{-1}(S)$ is a set of orbit representatives for $H^{(\bar{g},\sigma)} \under [n]$.
Thus, we have shown that $\beta$ is well-defined.

The fact that $\beta$ preserves the augmentation i.e.\ that $\|\beta([H \subseteq G \wr \Sigma_n])\| = n$ follows from the fact that 
\[ \sum_{Hi \in H\under [n]} |Hi| = n. \]

To see that $\beta$ is a monoid homomorphism, consider $[H] \in \Conj(G \wr \Sigma_i)$ and $[K] \in \Conj(G \wr \Sigma_j)$ where $n=i+j$.
Recall that $[H] \star [K] = [H \times K] \in \Conj(G \wr \Sigma_n)$.
We may choose $[n] = [i] \amalg [j]$ so that $(H \times K) \under [n] \cong (H \under [i]) \amalg (K \under [j])$.
Additionally, we may recognize that
\[ (H \times K)_{k\ell} = \begin{cases} H_{k\ell}, & k,\ell \in [i] \\ K_{k\ell}, & k,\ell \in [j] \\ \emptyset, &\text{otherwise} \end{cases} \]
for $k,\ell \in [n]$.
Then,
\begin{align*}
    \beta([H] \star [K]) 
    &= \sum_{(H \times K)k \in (H \times K)\under [n]} ([(H \times K)_{kk}], |(H \times K)k|)
    \\&= \sum_{Hk \in H\under [i]} ([H_{kk}],|Hk|) + \sum_{K\ell \in K \under [j]} ([K_{\ell\ell}],|K\ell|)
    \\&= \beta([H]) + \beta([K]).
\end{align*}

To see that $\beta$ is a retract of $\alpha$, consider some $\lambda \in \Parts(G,n)$.
As described in \ref{sec:alpha}, $\alpha(\lambda)$ is the conjugacy class 
\[ \alpha(\lambda) = \bigstar_{[H],m} [H \wr \Sigma_m]^{\star \lambda_{[H],m}} = \left[ \prod_{[H],m} (H \wr \Sigma_m)^{\times \lambda_{[H],m}} \right]. \]
For the sake of this proof, let $[n]$ be the set
\[ \coprod_{[H],m} [\lambda_{[H],m}] \times [m] = \{ ([H],m,i,j) \mid [H] \in \Conj(G), m \in \N_+, i \in [\lambda_{[H],m}], j \in [m] \}. \]
This determines the embedding $\prod_{[H],m} (G \wr \Sigma_m)^{\times \lambda_{[H],m}} \subseteq G \wr \Sigma_n$ as in \ref{subsec:fin sets}.
Let $P$ be the product subgroup
\[ P = \prod_{[H],m} (H \wr \Sigma_m)^{\times \lambda_{[H],m}} \subseteq \prod_{[H],m} (G \wr \Sigma_m)^{\times \lambda_{[H],m}} \subseteq G \wr \Sigma_n \]
so that $\alpha(\lambda) = [P]$.
Then, the $P$-orbits of $[n]$ are the sets $\{([H],m,i)\} \times [m] \subseteq [n]$ for $[H] \in \Conj(G)$, $m \in \N_+$, and $i \in [\lambda_{[H],m}]$
For $k=([H],m,i,j)$ a representative of such an orbit, the associated subgroup $P_{kk}$ is $H$.
Thus, it follows that
\begin{align*}
    \beta(\alpha(\lambda))
    &= \beta([P])
    = \sum_{\substack{P([H],m,i,j) \in P \under [n]}} ([P_{kk}], |Pk|) 
    \\&= \sum_{\substack{[H],m,i \\ i \in [\lambda_{[H],m}]}} ([H], m) 
    = \sum_{\substack{[H],m}} \lambda_{[H],m} \cdot ([H], m) 
    = \lambda,
\end{align*}
where $k$ in the first sum denotes $([H],m,i,j)$.
This completes the proof.
\end{proof}

\begin{proposition} \label{prop:alphabetagroup}
Given $H \subseteq G \wr \Sigma_n$, there is a unique representative $H' \in \alpha \beta([H])$ such that $H \subseteq H'$.
\end{proposition}
Note that we only rely on the existence of this representative in the remainder of the document.
\begin{proof}
    First we show existence. 
    Let $H'$ denote the subgroup 
    \[ H' = \{ (\bar g,\sigma) \in G \wr \Sigma_n \mid \bar g_{\sigma(i)} \in H_{i \sigma(i)} \ \forall i \in [n] \}  \subseteq G \wr \Sigma_n. \]
    The fact that this is a subgroup follows from the groupoid-like properties satisfied by the $H_{ij}$'s, and the fact that $H \subseteq H'$ follows directly from the definition of the $H_{ij}$'s.
    It also follows that $H'_{ij} = H_{ij}$ for any $i,j \in [n]$ and that the orbit $H'i$ equals the orbit $Hi$ for $i \in [n]$.
    Accordingly $H' \under [n] \cong H \under [n]$.

    Choose a section $s\colon H\under [n] \to [n]$ of the quotient $q \colon [n] \to H \under [n]$ and choose any $\bar g \in G^{\times n} \subseteq G \wr \Sigma_n$ so that $\bar g_i \in H_{sq(i) i}$.
    Then, observe that
    \[ H^{\bar g}{}_{ij} = \bar g_j^{-1} H_{ij} \bar g_i = H_{sq(i)sq(j)} = \begin{cases}
        H_{sq(i) sq(i)}, &\text{if } Hi=Hj \\
        \emptyset, &\text{otherwise.}
    \end{cases} \]
    Let $S = sq([n])$ be the collection of representatives of the orbits $H \under [n]$ determined by the section $s$.
    It follows that
    \[ H'^{\bar g} = \prod_{i \in S} H_{ii} \wr \Sigma_{Hi} \subseteq G \wr \Sigma_n, \]
    where $\Sigma_{Hi}$ denotes the symmetric group on the orbit $Hi \subseteq [n]$ and using the embedding $\prod_{i \in S} G \wr \Sigma_{Hi} \subseteq G \wr \Sigma_n$ arising from the bijection $\coprod_{i \in S} Hi \cong [n]$.
    This establishes that $[H'] = \alpha\beta([H]) = \alpha\beta([H'])$.

    To see that $H'$ is unique, recognize that any representative $H'' \in \alpha\beta([H])$ necessarily has the form
    \[ H'' = \{ (\bar g,\sigma) \in G \wr \Sigma_n \mid \bar g_{\sigma(i)} \in H''_{i \sigma(i)} \ \forall i \in [n] \} \]
    because $H'$ has this form and $H''$ is conjugate to $H'$.
    Thus, in order for $H''$ to contain $H$, it must be that $H_{ij} \subseteq H''_{ij}$ for all $i,j \in [n]$ and thus that $H_{ij} = H''_{ij}$ for all $i,j$, so that $H' = H''$.
\end{proof}

For completeness, we describe a few more ways to think about the map $\beta$. Using the notation established above, for $i \in S$ and $H \subseteq G \wr \Sigma_n$, we have 
\[
\Stab_H(i) \subseteq \Stab_{G \wr \Sigma_n}(i) \cong G \wr \Sigma_{\{i\}} \times G \wr \Sigma_{[n] - \{i\}} \cong G \times G \wr \Sigma_{n-1},
\]
where $[n]-\{i\}$ denotes the set difference.
Then $H_i = \pi_{\{i\}}(\Stab_H(i)) \subseteq G$ is another description of $H_{ii}$, where $\pi_{\{i\}}$ denotes projection onto $G \wr \Sigma_{\{i\}} \cong G$.

The data of a $\Conj(G)$-decorated partition of size $n$ is equivalent to an equivalence class of pairs $(X, f \colon [n] \twoheadrightarrow X/G)$, where $X$ is a finite right $G$-set, $X/G$ is the set of right orbits whose elements are of the form $xG$ for $x \in X$, and $f$ is a surjection (of sets).
The associated $\Conj(G)$-decorated partition of size $n$ is given by
\[ (X,f) \rightsquigarrow \sum_{xG \in X/G} ([\Stab_G(x)], |f^{-1}(xG)|). \]
Equivalently,
\[ \lambda_{[H],m} = \left|\left\{ xG \in X/G \colon |f^{-1}(xG)| = m \text{ and } [\Stab_G(x)] = [H]\right\}\right|. \]
To construct the $(X,f)$ pair corresponding to $\beta([H])$ for $H \subseteq G \wr \Sigma_n$, recall that $G \wr \Sigma_n \cong \Aut_G([n] \times G)$ where $[n] \times G$ is regarded as a right $G$-set by multiplication in the second factor.
Then, $[n] \times G$ is a $(H,G)$-biset, and the $(X,f)$ pair corresponding to $\beta([H])$ is 
\[ \beta([H]) \leftsquigarrow \bigg( H \under ([n] \times G),\ [n] \cong ([n] \times G)/G \twoheadrightarrow H \under ([n] \times G) / G \bigg). \]

Finally, note that the augmented monoid homomorphism $\beta$ induces a map of graded monoid rings $\Z[\Conj(G \wr \Sigma)] \to \Z[\Parks(G)]$.
Under the identifications $\Z[\Conj(G \wr \Sigma)] \cong A(G \wr \Sigma)$ and $\Z[\Parks(G)] \cong \AA(G)$, this corresponds to a graded ring map $r\colon A(G \wr \Sigma) \to \AA(G)$.
\begin{definition} \label{def:r}
    Let $r \colon A(G \wr \Sigma) \to \AA(G)$ be the commutative graded ring homomorphism defined on additive generators by
    \[ r \colon A(G \wr \Sigma) \to \AA(G) \colon [(G \wr \Sigma_n) / H] \mapsto [(G \wr \Sigma_n) / \alpha\beta([H])] = [(G \wr \Sigma_n)/H'], \]
    where $H' \in \alpha\beta([H])$.
    Equivalently, 
    \[ r \colon [(G \wr \Sigma_n) / H] \mapsto \bigstar_{[K],m} \P_m([G/K])^{\star \beta([H])_{[K],m}} = \bigstar_{Hi \in H\under [n]} \P_{|Hi|}([G/H_{ii}]). \]
\end{definition}
\begin{corollary}
The map $r \colon A(G \wr \Sigma) \to \AA(G)$ is a retract to the inclusion $\AA(G) \hookrightarrow A(G \wr \Sigma)$.
\end{corollary}

\section{Fixed points of submissive sets} \label{sec:fixedpoints}

In this section, we study the fixed point sets of $G \wr \Sigma_n$-sets of the form $X^{\times n}$ for a finite $G$-set $X$, and more generally of submissive $G \wr \Sigma_n$-sets, for any subgroup $H \subseteq G \wr \Sigma_n$.
These fixed point sets are closed related to the map $\beta$ of the previous section, and we continue to use the notation of that section.

In particular, any subgroup $H \subseteq G \wr \Sigma_n$ has a left action on $[n]$ via the quotient $G \wr \Sigma_n \twoheadrightarrow \Sigma_n$. 
Accordingly, we write $H \under [n]$ (read $H$ under $[n]$) for the collection of orbits of $[n]$ under the left action of $H$. 
Elements of $H \under [n]$ are of the form $Hi$ for $i \in [n]$.
Further, we make use of the previously defined sets $H_{ij}$ for $i,j \in [n]$.

As usual, $[n]$ is the disjoint union of its $H$-orbits.
Explicitly, 
\[ [n] \cong \coprod_{Hi \in H \under [n]} Hi. \]
We write $X^{\times Hi} = \prod_{j \in Hi} X \cong X^{\times |Hi|}$ and $\Sigma_{Hi}$ for the group of bijections of the set $Hi \subseteq [n]$.
With this notation, there is a canonical isomorphism
\[ \prod_{Hi \in H \under [n]} X^{\times Hi} \cong X^{\times [n]} = X^{\times n} \text{ and embedding } \prod_{Hi \in H \under[n]} \Sigma_{Hi} \subseteq \Sigma_{[n]} = \Sigma_n. \]
Then, it follows that
\[ H \subseteq \prod_{Hi \in H \under[n]} G \wr \Sigma_{Hi} \subseteq G \wr \Sigma_n. \]
Further, let 
\[ \pi_{Hj} \colon \prod_{Hi \in H \under[n]} G \wr \Sigma_{Hi} \to G \wr \Sigma_{Hj} \]
denote projection onto the $Hj$ factor for $Hj \in H \under [n]$.

There are three main propositions that we use to characterize the fixed point sets of submissive sets.
The other interesting facts are corollaries.

\begin{proposition} \label{prop:fixedpointsproduct}
    Let $X$ be a finite $G$-set and $H \subseteq G \wr \Sigma_n$. 
    Then, under the notation above, there is a bijection
    \[ (X^{\times n})^H \cong \prod_{Hi \in H\under [n]} (X^{\times Hi})^{\pi_{Hi} H}. \]
\end{proposition}
\begin{proof}
    This follows immediately from the fact that if $K$ is a subgroup of a product group $G \times H$ and $X$ and $Y$ are $G$ and $H$-sets respectively, then
    \[ (X \times Y)^K \cong X^{\pi_1K} \times Y^{\pi_2K} \]
    where $\pi_1$ and $\pi_2$ denote the projections onto $G$ and $H$ respectively.
\end{proof}

\begin{proposition} \label{cor:alphabetafixedpoints}
    Let $X \hookrightarrow Y^{\times n}$ be a submissive $G \wr \Sigma_n$-set and $H \subseteq G \wr \Sigma_n$.
    Then,
    \[ X^H = X^{H'}, \]
    where $H'$ is the unique representative of the conjugacy class $\alpha\beta([H])$ such that $H \subseteq H'$ as in \ref{prop:alphabetagroup}.
\end{proposition}
\begin{proof}
    It suffice to prove the claim for a submissive $G \wr \Sigma_n$-set $X \subseteq Y^{\times n}$.
    The idea is that, because the elements of $X$ are $n$-tuples, fixed points can be checked `component-wise.'
    That is,
    \begin{align*}
        X^H 
        &= \{ x \in X \mid (\bar g,\sigma)x = x \ \forall (\bar g,\sigma) \in H \}
        \\&= \{ x \in X \mid \bar g_{\sigma(i)} x_i = x_{\sigma(i)} \ \forall (\bar g,\sigma) \in H \ \forall i \in [n] \}
        \\&= \{ x \in X \mid g x_i = x_j \ \forall i,j \in [n] \ \forall g \in H_{ij} \}.
    \end{align*}
    where the sets $H_{ij}$ are defined as in the previous section.
    From the proof of \ref{prop:alphabetagroup}, $H'_{ij} = H_{ij}$ for all $i,j$.
    The result follows.
\end{proof}

\begin{corollary} \label{cor:submissivefixedpoints}
    If $X$ is submissive and $H'$ is any representative of $\alpha\beta([H])$, then $X^H \cong X^{H'}$.
\end{corollary}

We say that a subgroup $H \subseteq G \wr \Sigma_n$ is a transitive if the induced action of $H$ on $[n]$ is transitive. 
By construction, the subgroups $\pi_{Hi} H \subseteq G \wr \Sigma_{Hi}$ of \ref{prop:fixedpointsproduct} are transitive. 

\begin{proposition} \label{prop:transitivefixedpoints}
    Let $X$ be a finite $G$-set, $H \subseteq G \wr \Sigma_n$ be a transitive subgroup, and $i \in [n]$.
    Then, the projection $\pi_i \colon X^{\times n} \to X$ induces a bijection
    \[ (X^{\times n})^H \cong X^{H_{ii}}. \]
\end{proposition}
\begin{proof}
    As in the proof of the previous proposition,
    \[ (X^{\times n})^H = \{ x \in X^{\times n} \mid g x_j = x_k \ \forall j,k \in [n] \ \forall g \in H_{jk} \}. \]
    With this in mind, it is clear that the projection $\pi_i \colon (X^{\times n})^H \to X$ lands in the desired fixed points $X^{H_{ii}}$.
    To see that this is a bijection, choose $g_j \in H_{ij}$ for each $j \in [n]$, which is possible since $H$ is transitive.
    Define the inverse of the projection by $x \in X^{H_{ii}} \mapsto y(x) \in X^{\times n}$ where $y(x)_j$ is defined to be $g_jx$.
    This lands in the $H$-fixed points as desired because for $h \in H_{jk}$, 
    \[ hy(x)_j = hg_j x = g_k g_k^{-1}hg_j x = g_k x = y(x)_k, \]
    since $g_k^{-1}hg_j \in H_{ii}$ and $H_{ii}$ fixes $x$.
\end{proof}

\begin{corollary} \label{cor:fixedpointpower}
    Let $X$ be a finite $G$-set and $H \subseteq G \wr \Sigma_n$.
    There is a bijection
    \[ (X^{\times n})^H \cong \prod_{Hi \in H \under [n]} X^{H_{ii}}. \]
\end{corollary}
\begin{proof}
    This follows from \ref{prop:fixedpointsproduct} and \ref{prop:transitivefixedpoints} together with the fact that $(\pi_{Hi} H)_{ii} = H_{ii}$.
\end{proof}

\begin{corollary} \label{cor:partitionfixedpoints}
    Let $X$ be a finite $G$-set and $\lambda \in \Parts(G,n)$. Choose a representative $H \in \alpha(\lambda)$.
    Then,
    \[ |(X^{\times n})^H| = \prod_{[K],m} |X^K|^{\lambda_{[K],m}}. \]
\end{corollary}
\begin{proof}
    This follows from \ref{cor:fixedpointpower} and \ref{def:alpha} which defines $\alpha$.
\end{proof}

\section{Decorated compositions} \label{sec:decoratedcompositions}

Much like it is often convenient to choose representatives of conjugacy classes of subgroups, it is often useful to choose finer objects representing partitions.
This is the notion a composition, which refers to an indexed or ordered formal sum of integers.

Accordingly, a formal sum in $X \times \N_+$ indexed by a finite set $I$ is called an $I$-indexed $X$-decorated composition, and the collection of such compositions is denoted $C(I,X) = (X \times \N_+)^I$.
For $c \in C(I,X)$ and $i \in I$, the $i$-th pair in $c$ is denoted $c_i$ and is called the $i$-th part of $c$.
We denote the components of these parts as $c_i=(c_i^\dec, c_i^\num)$ so that $c_i^\dec$ denotes the decoration and $c_i^\num$ denotes the integer component.
If $c \in C(I,*)$ is an integer composition, we may write $c_i$ for $c_i^\num$ since the decoration contributes no information.

There is a function 
\[ u \colon C(I,X) \to P(X) \colon c \mapsto uc = \sum_{i \in I} c_i \]
that un-indexes a composition by taking its formal sum.
Equivalently,
\[ uc \colon x,m \mapsto |\{ i \in I \mid c_i = (x,m) \}|. \]
We denote the preimage under $u$ of a partition $\lambda$ by $C(I,\lambda)$ and we call the elements of this set $I$-indexed compositions of $\lambda$.
The set $C(I,\lambda)$ is nonempty if and only if $|I| = |\lambda|$.
Compositions are generally denoted by lowercase Latin letters toward the beginning of the alphabet, but sometimes notation will be chosen to match corresponding partitions e.g.\ we may choose $\ell$ for an element of $C(I,\lambda)$.
Certain expressions involving partitions are related to certain expressions related to compositions.
In particular, if $\lambda \in P(X)$, $\ell \in C(I,\lambda)$, and $f \colon X \times \N_+ \to M$ is a function into a commutative monoid, then
\begin{equation} 
\begin{gathered} \label{eq:independent}
    \sum_{x,m} \lambda_{x,m} \cdot f(x,m) = \sum_i f(\ell_i) \text{ or } \prod_{x,m} f(x,m)^{\lambda_{x,m}} = \prod_i f(\ell_i)
    \\ \text{ or } \bigstar_{x,m} f(x,m)^{\star \lambda_{x,m}} = \bigstar_i f(\ell_i) 
\end{gathered}
\end{equation}
depending on which symbol is used for the operation of $M$.

We extend the operations of partitions to compositions such that operations applied to compositions correspond to operations applied to the underlying partitions.
Addition of partitions corresponds to a concatenation map $+ \colon C(I,X) \times C(J,X) \to C(I\amalg J,X)$ defined by
\[ (c+d)_i = \begin{cases} c_i,& i \in I \\ d_i,& i \in J, \end{cases} \]
for $c \in C(I,X)$ and $d \in C(J,X)$.
For $J \subseteq I$, there is a restriction map $C(I,X) \to C(J,X)\colon c \mapsto c|_J$ which remembers only the $i$-th parts of $c$ for $i \in J$.
This interacts with concatenation so that $c = c|_I + c|_J$ for $c \in C(I \amalg J,X)$.
Consequently the concatenation map is a bijection.
Concatenation of compositions corresponds to addition of partitions in the sense that $u(c+d)=uc+ud$.
Accordingly, for $\lambda,\mu \in P(X)$, the concatenation map restricts to a map $C(I,\lambda) \times C(J,\mu) \to C(I \amalg J,\lambda + \mu)$ but this is not usually bijection, even when the domain and codomain are nonempty.

We define the length and size of a composition $c \in C(I,X)$ by
\[ |c| = |I| = |uc| \text{ and } \|c\| = \sum_i c_i^\num = \|uc\| \]
and the factorial and multifactorial by
\[ c! = (uc)! \text{ and } \mf c = \prod_i c_i^\num! = \mf{uc}. \]
Given a function $f \colon X \to Y$, there is a pushforward map $f_* \colon C(I,X) \to C(I,Y)$ which applies $f$ to the decorations i.e.\ $f_*c_i = (f(c_i^\dec),c_i^\num)$.
This has the property that $f_*uc = uf_*c$.
The pushforward along the map $X \to *$ is the un-decoration map $C(I,X) \to C(I,*) \colon c \mapsto c^\num$.
Similarly, we denote the un-decoration map $P(X) \to P(*) \colon \lambda \mapsto \lambda^\num$ so that
\[ \lambda^\num \colon m \mapsto \sum_x \lambda_{x,m}. \]

When we do not care about the particular set which indexes compositions, we write $C(k,X)$ for $C([k],X)$ where $k \in \N$.
Similarly, we write $C(\lambda)$ for $C(|\lambda|,\lambda)$ and call the elements of this set compositions of $\lambda$.
Sometimes, it is useful to refer to the collection of $X$-decorated compositions of any length.
We denote this set by $C(X) = \coprod_k C(k,X)$.
Concatenation gives $C(X)$ the structure of a monoid, using the identifications $[i] \amalg [j] \cong [i+j]$ described in \ref{subsec:fin sets}.
This monoid is not commutative when $X$ is nonempty, but the un-indexing map is a homomorphism $u \colon C(X) \to P(X)$ witnessing $P(X)$ as the abelianization of $C(X)$.
The map $u \colon C(X) \to P(X)$ is a homomorphism of augmented monoids when $C(X)$ is augmented by the size map (as opposed to the augmentation $c \in C(k,X) \mapsto k$ given by the disjoint union decomposition).
For $c \in C(X)$ and $f \colon X \times \N_+ \to M$ a function into a commutative monoid, the only reasonable interpretation of 
\[ \sum_i f(c_i) \text{ is } \sum_{i \in [|c|]} f(c_i), \]
where $[|c|]$ is the set of size $|c|$, so we usually opt to elide the domains of such sums (or products or $\star$-products).
We may identify $X \times \N_+$ with $C(1,X)$ so that for $c \in C(X)$, $x \in X$, and $m \in \N_+$, $c+(x,m)$ denotes the composition obtained by appending the part $(x,m)$ to the end of $c$.
If $X=*$, we write $c+m$ for the integer composition $c$ with $m$ appended to the end.

It is often useful to choose a composition for a specific partition.
Given a partition $\lambda$, we denote such an unspecified choice by $c(\lambda) \in C(\lambda)$.
The particular values $c(\lambda)_i$ are not defined, but regardless certain expressions in terms of them are still defined.
For example, if $f \colon X \times \N_+ \to M$ is any function into a commutative monoid, then
\begin{gather*}
    \sum_i f(c(\lambda)_i) = \sum_{x,m} \lambda_{x,m} \cdot f(x,m) \text{ or } \prod_i f(c(\lambda)_i) = \prod_{x,m} f(x,m)^{\lambda_{x,m}}
    \\ \text{ or } \bigstar_i f(c(\lambda)_i) = \bigstar_{x,m} f(x,m)^{\star \lambda_{x,m}}
\end{gather*}
is well-defined depending on which symbol is used for the operation of $M$.

Notice that $C(I,X)$ has an action of $\Sigma_I = \Aut(I)$.
For $c \in C(I,X)$, we denote the associated stabilizer by 
\[ \Sigma_c = \{ \sigma \mid \sigma c = c \} \subseteq \Sigma_I. \]
Then, $c! = |\Sigma_c|$.
The collection $C(\lambda)$ is the $\Sigma_{|\lambda|}$-orbit of $c(\lambda) \in C(|\lambda|)$. 
Accordingly, we find that 
\[ \frac{|\lambda|!}{\lambda!} = [\Sigma_{|\lambda|} \colon \Sigma_{c(\lambda)} ] = |C(\lambda)|. \]
The canonical embedding $\Sigma_I \times \Sigma_J \subseteq \Sigma_{I \amalg J}$ gives a relationship between the stabilizers:
$\Sigma_c \times \Sigma_d \subseteq \Sigma_{c+d}$.
This witnesses that
\[ \frac{(c+d)!}{c!d!} = [\Sigma_{c+d} \colon \Sigma_c \times \Sigma_d]. \]
As mentioned, the pushforward along a function $f\colon X \to Y$ does not generally preserve the factorial, but this failure has combinatorial meaning:
$\Sigma_c \subseteq \Sigma_{f_*c}$, and
\[ \frac{f_*\lambda!}{\lambda!} = \frac{(f_*\lambda)!}{\lambda!} = [\Sigma_{f_*c} \colon \Sigma_c]. \]

\section{Examples of partition and composition notation} \label{sec:partexamples}

We dedicate this section to example applications of partitions and compositions.
The results in this section are either reformulations of existing results into new language or are elementary results that can be efficiently stated in terms of partitions and compositions.
We invite the reader to use this section to internalize the notation of the previous section, as it will be heavily used in the remainder of the paper.

\subsubsection*{The multinomial theorem}

The multinomial theorem
\[ \left( \sum_{k=1}^m x_k \right)^n = \sum_{\substack{i_1+\cdots+i_m=n \\ i_1,\dots,i_m \geq 0}} \frac{n!}{\textstyle \prod_k i_k!} \prod_k x_k^{i_k} \]
can be expressed in terms of integer compositions.
An equivalent formulation is
\[ \left( \sum_{i \in I} x_i \right)^n = \sum_{\substack{\|c\|=n \\ J \subseteq I,\, c \in C(J,*)}} \frac{n!}{\mf c} \prod_{i \in J} x_i^{c_i} = \sum_{\substack{\|c\|=n \\ J \subseteq I,\, c \in C(J,*)}} \frac{n!}{\mf c} x^c, \]
where $x^c$ denotes the quantity $\prod_{i \in J} x_i^{c_i}$.
Another equivalent formulation involves integer partitions:
\[ \left( \sum_{i \in I} x_i \right)^n = \sum_{\substack{\|\lambda\| = n \\ \lambda \in P(*)}} \frac{n!}{\mf \lambda} \sum_{\substack{\ell \in C(J,\lambda) \\ J \subseteq I,\, |J|=|\lambda|}} x^\ell. \]
The quantities $n!/\mf\lambda= \|\lambda\|!/\mf\lambda$ are called the multinomial coefficients and are integral.
The necessity of summing over subsets $J \subseteq I$ is an artifact of our definition of integer compositions having positive parts.
While other choices may result in more elegant formulas, we choose this formulation to avoid introducing too many variants of compositions and partitions.

\subsubsection*{The multinomial theorem for power operations}

The notion of $R$-valued power operations on $A$ for $R$ a commutative graded ring and $A$ an abelian group as defined in \ref{sec:universalpowerops} has three axioms: $P_0(x)=1$, $P_n(0) = 0$ for $n>0$ and
\[ P_n(x+y)=\sum_{i+j=n} P_i(x)P_j(y). \]
From this, one can conclude a version of the multinomial theorem without the coefficients:
\[ P_n\left(\sum_{i \in I} x_i\right) = \sum_{\substack{\|c\| = n \\ J \subseteq I,\, c \in C(J,*)}} \prod_{i \in J} P_{c_i}(x_i) = \sum_{\substack{\|c\| = n \\ J \subseteq I,\, c \in C(J,*)}} \hspace{-1em} P_c(x) =  \sum_{\substack{\|\lambda\| = n,\, \ell \in C(J,\lambda) \\ \lambda \in P(*),\, J \subseteq I,\, |J|=|\lambda|}} \hspace{-2em} P_\ell(x), \]
where $P_c(x)$ denotes $\prod_{i \in J} P_{c_i}(x_i)$.
This is essentially a rewriting of \eqref{eq:power op on a sum} in the language of partitions and compositions.

The proof of this formula is the same as the inductive proof of the classical multinomial theorem.
This single formula implies each of the other axioms: $P_0(x)=1$ follows from $n=0$ and $|I|=1$, $P_n(0)=0$ for $n>0$ follows from $|I|=0$, and $P_n(x+y)=\sum_{i+j=n}P_i(x)P_j(y)$ follows from $|I|=2$.
Thus, taking this as the only axiom provides an equivalent definition for power operations.

\subsubsection*{Power operations of additive inverses} 

The language of partitions allows an elegant formula for power operations applied to additive inverses.
Let $P$ be an $R$-valued family of power operations on $A$.
For an integer partition $\lambda \in P(*)$ and $x \in A$, let $P_\lambda(x) = \prod_k P_k(x)^{\lambda_k} \in R_{\|\lambda\|}$.
Similarly, for an integer composition $c \in C(*)$, let $P_c(x) = \prod_i P_{c_i}(x) = P_{uc}(x)$.
Note that this agrees with the notation $P_c(x)$ of the previous subsection when $(x_i)_{i \in J}$ is a constant family.
\begin{proposition}
    \[ P_n(-x) = \sum_{\substack{\|\lambda\| = n \\ \lambda \in P(*)}} (-1)^{|\lambda|} \frac{|\lambda|!}{\lambda!} P_\lambda(x). \]
\end{proposition}
\begin{proof}
    This can be proved by induction.
    First, note that
    \[ P_n(0) = \sum_{\substack{\|\lambda\| = n \\ |\lambda| = 0}} (-1)^{|\lambda|} \frac{|\lambda|!}{\lambda!} P_\lambda(x) \]
    because the value of the summand is $1$ when $\lambda=0$, and the sum is empty of $n>0$.
    Suppose for the sake of induction that the formula of the proposition holds for $k<n$.
    Then, 
    \begin{align*}
        P_n(0) 
        &=P_n(x-x) 
        = \sum_{i+j=n} P_i(x)P_j(-x)
        = P_n(-x) + \sum_{\substack{i+j=n \\ i > 0}} P_i(x)P_j(-x)
        \\&= P_n(-x) + \sum_{\substack{i+j=n \\ i > 0}} P_i(x)\sum_{\substack{\|\lambda\| = j \\ \lambda \in P(*)}} (-1)^{|\lambda|} \frac{|\lambda|!}{\lambda!} P_\lambda(x)
        \\&= P_n(-x) + \sum_{\substack{i+j=n \\ i > 0}} P_i(x)\sum_{\substack{\|\ell\| = j \\ \ell \in C(*)}} (-1)^{|\ell|} P_\ell(x)
        \\&= P_n(-x) - \sum_{\substack{i+j=n \\ i > 0}} \sum_{\substack{\|\ell\| = j \\ c = \ell + i}} (-1)^{|c|} P_c(x)
        \\&= P_n(-x) - \sum_{\substack{\|c\| = n \\ |c|>0}} (-1)^{|c|} P_c(x)
        \\&= P_n(-x) - \sum_{\substack{\|\lambda\| = n \\ |\lambda|>0}} (-1)^{|\lambda|} \frac{|\lambda|!}{\lambda!} P_\lambda(x).
    \end{align*}
    The equality between the second and third lines follows by reparameterizing from partitions to compositions.
    The equality between the third and fourth lines follows by rewriting in terms of the composition $c=\ell+i$.
    The equality between the fifth and sixth lines follows by reparameterizing back to partitions.
    We find that
    \begin{align*}
        P_n(-x) 
        &= P_n(0) + \sum_{\substack{\|\lambda\| = n \\ |\lambda|>0}} (-1)^{|\lambda|} \frac{|\lambda|!}{\lambda!} P_\lambda(x)
        = \sum_{\substack{\|\lambda\| = n \\ |\lambda|=0}} (-1)^{|\lambda|} \frac{|\lambda|!}{\lambda!} P_\lambda(x) + \sum_{\substack{\|\lambda\| = n \\ |\lambda|>0}} (-1)^{|\lambda|} \frac{|\lambda|!}{\lambda!} P_\lambda(x)
        \\&= \sum_{\|\lambda\| = n} (-1)^{|\lambda|} \frac{|\lambda|!}{\lambda!} P_\lambda(x).
        \qedhere
    \end{align*}
\end{proof}
\begin{corollary} \label{cor:powofdiff} There is an equality
    \[ P_n(x-y) = \sum_{i+\|\lambda\|=n} (-1)^{|\lambda|} \frac{|\lambda|!}{\lambda!} P_i(x) P_\lambda(y). \]
\end{corollary}

\begin{corollary} \label{cor:pos enough}
    Any summand of $A(G \wr \Sigma_n)$ spanned by a subset of the canonical basis and containing the image of $\P_n \colon A^+(G) \to A(G \wr \Sigma_n)$ also contains the image of $\P_n \colon A(G) \to A(G \wr \Sigma_n)$.
\end{corollary}
\begin{proof}
    This follows from the fact that each of the summands $P_i(x)P_\lambda(y)$ appearing in the formula for $P_n(x-y)$ is one of the summands in the expansion of $P_n(x+|\lambda|y)$.
\end{proof}

In the case of the Burnside ring (or other similar cases), the power operations $\P_n \colon A(G) \to \AA(G,n)$ are multiplicative (with the ring structure internal to $\AA(G,n)$, not the transfer product) so that $\P_n(-X) = \P_n(-1)\P_n(X)$.
In this case, it is only necessary to evaluate this formula to compute $\P_n(-1)$.
\begin{corollary} The following equation holds
    \[ \P_n(-1) = \sum_{\|\lambda\| = n} (-1)^{|\lambda|} \frac{|\lambda|!}{\lambda!} \P_\lambda(1) = \sum_{\|\lambda\| = n} (-1)^{|\lambda|} \frac{|\lambda|!}{\lambda!} \Res_{G \to e}[\Sigma_n / \alpha(\lambda)]. \]
\end{corollary}

\subsubsection*{Transfer products}

The transfer product on $\Parks(G,R)$, defined as in \ref{prop:parksprod} by
\[ (f\star g)(\mu) = \sum_{\kappa + \lambda = \mu} \frac{\mu!}{\kappa!\lambda!} f(\kappa)g(\lambda), \]
for $f,g \in \Parks(G,R)$ and $\mu \in \Parts(G)$, has an alternative description in terms of compositions.

\begin{proposition} \label{prop:compprod}
    For $f,g \in \Parks(G,R)$ and $\mu \in \Parks(G)$,
    \[ (f\star g)(\mu) = \sum_{I \cupdot J = [|\mu|]} f(uc(\mu)|_I) g(uc(\mu)|_J), \]
    where $[|\mu|]$ is the set of size $|\mu|$ and $\cupdot$ denotes the internal disjoint union of subsets of $[|\mu|]$.
    In particular, this sum ranges over $I,J \subseteq [|\mu|]$ such that $I \cap J = \emptyset$ and $I \cup J = I \cupdot J = [|\mu|]$.
    More generally, for $I$ a finite indexing set and $(f_i \in \Parks(G))_{i \in I}$ an $I$-indexed family of parks,
    \[ \left( \bigstar_{i \in I} f_i \right)(\mu) = \sum_{\bigcupdot_{i\in I} J_i = [|\mu|]} \prod_{i \in I} f_i(uc(\mu)|_{J_i}) = \sum_{g \colon [|\mu|] \to I} \prod_{i \in I} f_i(uc(\mu)|_{g^{-1}(i)}), \]
    where the first sum ranges over all pair-wise disjoint $I$-indexed covers $(J_i \subset [|\mu|])_{i \in I}$ of $[|\mu|]$.
\end{proposition}
\begin{proof}
    The formula follows from the observation that
    \[ \sum_{I \cupdot J = [|\mu|]} f(uc(\mu)|_I) g(uc(\mu)|_J) = \sum_{\kappa + \lambda = \mu} \left|\left\{ I \cupdot J = [|\mu|] \colon \substack{uc(\mu)|_I = \kappa \\ uc(\mu)|_J = \lambda} \right\}\right| f(\kappa)g(\lambda). \] 
    Then, note that a choice of $I \cupdot J = [|\mu|]$ such that $uc(\mu)|_I = \kappa$ and $uc(\mu)|_J = \lambda$ amounts to a choice of $\kappa_{x,m}$-many of the $i \in [|\mu|]$ such that $c(\mu)_i = (x,m)$ for each $x$ and $m$, which is enumerated by the binomial coefficient $\mu_{x,m}!/\kappa_{x,k}!\lambda_{x,m}!$.
\end{proof}

\begin{corollary}
    The parks homomorphism $\chi \colon \AA(G) \to \Parks(G)$ is given by
    \[ \chi([(G \wr \Sigma_n)/\alpha(\lambda)])(\kappa) = \sum_{\substack{g \colon [|\kappa|] \to [|\lambda|] \\ \| c(\kappa)|_{g^{-1}(i)} \| = c(\lambda)_i^\num }} \prod_{j \in [|\kappa|]} \chi([G/c(\lambda)_{g(j)}^\dec])(c(\kappa)_j^\dec) \]
    for $\kappa,\lambda \in \Parts(G)$, where $[G/c(\lambda)^\dec_i]$ denotes the transitive $G$-set corresponding to the conjugacy class $c(\lambda)^\dec_i \in \Conj(G)$ i.e.\ $[G/c(\lambda)^\dec_i] = [G/H]$ for any $H \in c(\lambda)^\dec_i$. 
    Note that this formula yields $0$ when $\kappa$ and $\lambda$ do not have the same size.
\end{corollary}
\begin{proof}
    This follows from \ref{prop:compprod}, \eqref{eq:alpha star prod}, and \ref{cor:trivial powopformula}.
\end{proof}
This formula can be stated more efficiently in terms of refinements of partitions, but we will not pursue this in order to avoid inflicting even more notation related to partitions and compositions on the reader.

\section{Character theory for $\AA(G)$} \label{sec:pullback}

In this section we describe a universal property of $\AA(G)$ in terms of the map $\beta$ and the parks homomorphism of \ref{def:aachar}.

Recall the commutative graded ring $\Marks(G \wr \Sigma) = \bigoplus_n \Marks(G \wr \Sigma_n)$ from \ref{subsec:cgrings} and the map $\beta \colon \Conj(G \wr \Sigma) \to \Parts(G)$ of \ref{def:beta}.
Precomposition with $\beta$ induces a map $\beta^* \colon \Parks(G,R) \to \Marks(G \wr \Sigma,R)$ of graded abelian groups given by
\[ \beta^*f([H]) = f(\beta([H])) \]
for $f \in \Parks(G,R)$ and $[H] \in \Conj(G \wr \Sigma)$.
Since $\beta$ is surjective, $\beta^*$ is injective.

\begin{proposition}
    The map $\beta^* \colon \Parks(G,R) \to \Marks(G \wr \Sigma,R)$ is a homomorphism of commutative graded rings.
\end{proposition}
\begin{proof}
    It suffices to verify the claim for $R=\Z$.
    It is immediate that $\beta^*$ preserves the $\star$-unit because $\beta([e \subseteq G \wr \Sigma_0])$ is the empty partition $0$.
    Thus, it suffices to show that $\beta^*$ preserves the transfer product.
    We will verify this on the graded pieces.
    
    Consider $i+j=n$, $f \in \Parks(G,i)$, $g \in \Parks(G,j)$, and $[K] \in \Conj(G \wr \Sigma_n)$.
    Then using the transfer product formula of \ref{subsec:marks tr prod},
    \[ (\beta^* f \star \beta^*g)([K]) = \sum_{\substack{\sigma(\Sigma_i \times \Sigma_j) \in \Sigma_n / \Sigma_i \times \Sigma_j \\ K^\sigma \subseteq G \wr \Sigma_i \times G \wr \Sigma_j}} f(\beta(\pi_{[i]}[K^\sigma]))g(\beta(\pi_{[j]}[K^\sigma])). \]

    Consider the condition $K^\sigma \subseteq G \wr \Sigma_i \times G \wr \Sigma_j$.
    This occurs precisely when each orbit in the quotient $K^\sigma \under [n]$ is completely contained in either $[i]$ or $[j]$. 
    Observe that the $K^\sigma$ orbits are exactly the images under $\sigma^{-1}$ of the $K$ orbits i.e.\ $\sigma^{-1}(Ka) = K^\sigma \sigma^{-1}(a)$ for $a \in [n]$. 
    Thus, such a $\sigma$ determines a decomposition $I \cupdot J = K \under [n]$, where $I = \{ Ka \mid \sigma^{-1}(Ka) \subseteq [i] \}$ and similarly for $J$. Note that $I$ and $J$ are sets of orbits.
    This decomposition has the property that $\sum_{Ka \in I} |Ka| = i$ and similarly for $J$.
    Further, such a decomposition uniquely determines the coset $\sigma(\Sigma_i \times \Sigma_j)$ because any two $\sigma$ resulting in the same decomposition differ by a rearrangement of the elements of $[i]$ and $[j]$ which does not swap elements between them.

    Let $k \in C(\Conj(G), K \under [n])$ be defined by $k_{Ka} = ([K_{aa}], |Ka|)$.
    By \ref{def:beta}, $\beta([K]) = uk$. 
    Now, we seek to write $\beta(\pi_{[i]}[K^\sigma])$ in terms of $k$.
    In this notation, $\pi_{[i]}(K^\sigma)_{aa} = K_{\sigma(a)\sigma(a)}$, and the orbit $\pi_{[i]}(K^\sigma) a \in \pi_{[i]}(K^\sigma) \under [i]$ is just $K^\sigma a = \sigma^{-1}(K \sigma(a))$ so that the set of orbits $\pi_{[i]}(K^\sigma) \under [i]$ is exactly the set $\sigma^{-1}(I)$.
    Thus, it follows that
    \begin{align*}
        \beta(\pi_{[i]}(K^\sigma)) 
        &= \sum_{\pi_{[i]}(K^\sigma) a \in \pi_{[i]}(K^\sigma) \under [i]} ([\pi_{[i]}(K^\sigma)_{aa}], |\pi_{[i]}(K^\sigma) a|) 
        \\&= \sum_{\pi_{[i]}(K^\sigma) a \in \pi_{[i]}(K^\sigma) \under [i]} ([K_{\sigma(a)\sigma(a)}], |K \sigma(a)|)
        \\&= \sum_{K a \in I} ([K_{aa}], |Ka|) = uk|_I
    \end{align*}
    and similarly that $\beta(\pi_{[j]}(K^\sigma)) = uk|_J$.
    Additionally, observe that the condition $\sum_{Ka \in I} |Ka| = i$ is equivalent to $\|k|_I\| = i$ and similarly for $j$.

    Assembling these observations, we find that
    \begin{align*}
        (\beta^* f \star \beta^*g)([K]) = \sum_{\substack{I \cupdot J = K \under [n] \\ \|k|_I\| = i }} f(uk|_I)g(uk|_J) = \sum_{I \cupdot J = K \under [n]} f(uk|_I)g(uk|_J)
    \end{align*}
    where the second equality follows from the fact that $f \in \Parks(G,i)$ is only supported on partitions of size $i$.
    Then, using \ref{prop:compprod}, it follows that
    \[ (\beta^* f \star \beta^*g)([K]) = (f \star g)(uk) = (f\star g)(\beta([K])) = \beta^*(f \star g)([K]). \qedhere \] 
\end{proof}

\begin{proposition} \label{prop:pullback}
    There is a pullback diagram
    \begin{equation} \label{eq:parkssquare}
    \begin{tikzcd}
        \AA(G) \ar[r] \ar[d,"\chi"'] & A(G \wr \Sigma) \ar[d, "\chi"] \\
        \Parks(G) \ar[r, "\beta^*"] & \Marks(G \wr \Sigma),
    \end{tikzcd}
    \end{equation}
    of commutative graded rings where $\chi \colon \AA(G) \to \Parks(G)$ is the map of \ref{def:aachar}.
\end{proposition}
\begin{proof}
    Let $\bar A(G)$ be the pullback.
    This may be identified with the subgroup $\bar A(G) \subseteq A(G \wr \Sigma)$ whose characters in $\Marks(G \wr \Sigma)$ lie in the image of $\beta^*$.
    Since $\beta$ has a section $\alpha$, the image of $\beta^*$ is exactly the collection of $f \in \Marks(G \wr \Sigma)$ such that $f([H]) = f(\alpha\beta([H]))$ for each $[H] \in \Conj(G \wr \Sigma)$.
    By \ref{cor:submissivefixedpoints}, every isomorphism class of a submissive $G \wr \Sigma_n$-set is in $\bar{A}(G)$. 
    This provides an inclusion $\AA(G) \subseteq \bar A(G)$.

    Since $\Q$ is flat and $\Q \tensor A(G\wr \Sigma) \cong \Marks(G\wr \Sigma,\Q)$, it follows that $\Q \tensor \bar A(G) \cong \Parks(G,\Q)$.
    Since $\bar A(G)$ is a graded submodule of the free graded $\Z$-module $A(G \wr \Sigma)$, it follows that $\bar A$ is a free graded $\Z$-module on $|{\Parks(G,n)}|$-many generators of degree $n$.
    Since $\bar A(G)$ contains the subset of the generators of $A(G \wr \Sigma)$ of the form $\alpha(\lambda)$ for $\lambda \in \Parts(G)$, it follows that $\bar A(G)$ is the summand $\AA(G)$.

    All that remains to show is that the map $\AA(G) \cong \bar A(G) \to \Parks(G)$ induced by the pullback is the map $\chi \colon \AA(G) \to \Parks(G)$ of \ref{def:aachar}, and for this it suffices to show that \eqref{eq:parkssquare} commutes.
    To show the square commutes, it suffices to check on the polynomial generators $\P_n([G/H])$ of \ref{cor:AAisPoly}.
    The fact that the square commutes on polynomial generators follows from \ref{def:alpha} of $\alpha$ and \ref{def:aachar} of $\chi \colon \AA(G) \to \Parks(G)$ together with \ref{cor:partitionfixedpoints}.
\end{proof}

\begin{corollary}
    The parks homomorphism $\chi \colon \AA(G) \to \Parks(G)$ is a graded subring embedding.
    In particular, the character maps $\AA(G,n) \to \Parks(G,n)$ are injective.
\end{corollary}
\begin{corollary}
    The parks homomorphism is rationally an isomorphism.
    That is,
    \[ \Q \otimes \chi \colon \Q \tensor \AA(G) \to \Q \tensor \Parks(G) \cong \Parks(G,\Q) \]
    is an isomorphism.
    In particular, the maps $\Q \otimes \AA(G,n) \to \Parks(G,n,\Q)$ are isomorphisms.
\end{corollary}

\begin{corollary} \label{cor:levelpullback}
    There is a pullback diagram
    \[\begin{tikzcd}
        \AA(G,n) \ar[r] \ar[d, "\chi"'] & A(G \wr \Sigma_n) \ar[d, "\chi"] \\
        \Parks(G,n) \ar[r, "\beta^*"] & \Marks(G \wr \Sigma_n)
    \end{tikzcd}\]
    of commutative rings.
    In particular, $\chi \colon \AA(G,n) \to \Parks(G,n)$ is a map of rings.
\end{corollary}
\begin{proof}
    The fact this is a pullback is immediate. 
    The fact that it is a pullback of rings follows from the fact that $\chi \colon A(G \wr \Sigma_n) \to \Marks(G \wr \Sigma_n)$ and $\beta^*\colon \Parks(G,n) \to \Marks(G \wr \Sigma_n)$ are ring homomorphisms.
\end{proof}

\begin{corollary} \label{cor:parksfrommarks}
    The parks homomorphism can be computed in terms of the marks homomorphism.
    Specifically, 
    \[ \chi([(G \wr \Sigma_n)/\alpha(\lambda)])(\kappa) = \chi([(G \wr \Sigma_n)/\alpha(\lambda)])(\alpha(\kappa)) \]
    for $\kappa,\lambda \in \Parts(G,n)$.
    Note that the $\chi$ on the left is the parks homomorphism while the $\chi$ on the right is the marks homomorphism.
\end{corollary}
\begin{proof}
    This follows from the fact that $\beta\alpha(\kappa)=\kappa$.
\end{proof}

Importantly, we may obtain a formula for the character of the total power operation on the Burnside ring.
A formula similar to this has appeared in the literature in \cite[Section 4.3.4]{Ganter}.
\begin{corollary}
    There are commutative diagrams
    \[\begin{tikzcd}
        A(G) \ar[r, "\P_n"] \ar[d, "\chi"] & \AA(G,n) \ar[r] \ar[d, "\chi"] & A(G \wr \Sigma_n) \ar[d, "\chi"] \\
        \Marks(G) \ar[r, "\P_n"] & \Parks(G,n) \ar[r, "\beta^*"] & \Marks(G \wr \Sigma_n),
    \end{tikzcd}\]
    where $\P_n \colon \Marks(G) \to \Parks(G,n)$ is the power operation of \ref{prop:powopformula}.
    In particular, this provides a formula for $\P_n(f) \in \Marks(G \wr \Sigma_n)$ for $f \in \Marks(G)$:
    \[ \chi(\P_n(f))([H]) = \prod_{[K],m} f([K])^{\beta([H])_{[K],m}} = \prod_{Hi \in H \under [n]} f([H_{ii}]) \]
    for $[H] \in \Conj(G \wr \Sigma_n)$.
\end{corollary}

\section{Formulas for induced operations on parks} \label{sec:parksformulas}

Let $G$ and $H$ be finite groups. We consider a fixed abelian group homomorphism $F \colon A(H) \to A(G)$. Let $\r F \colon \AA(H) \to \AA(G)$ be the map of \ref{def:Fang}.
Rationalizing, we obtain maps 
\[ F\colon \Marks(H,\Q) \to \Marks(G,\Q) \text{ and } \r F \colon \Parks(H,\Q) \to \Parks(G,\Q). \]
The map on marks can be given a formula using the canonical bases.
In this section, we provide a formula for the map $\r F$ on parks in terms of the canonical bases.

Let $M(F)$ denote the matrix representation of $F$ on marks in the bases of indicator functions.
That is,
\[ F(f)([K]) = \sum_{[L] \in \Conj(H)} M(F)^{[K]}_{[L]} f([L]), \]
where $f \in \Marks(H,\Q)$ and $[K] \in \Conj(G)$ or equivalently
\[ F(\1_{[L]}) = \sum_{[K]} M(F)^{[K]}_{[L]} \1_{[K]}, \]
where $\1_{[L]}$ denotes the indicator function of a conjugacy class $[L]$.

\begin{proposition} \label{prop:transfercoefs}
    Define the quantity $\r M(F)^\kappa_\lambda$ for $\kappa \in \Parts(G)$ and $\lambda \in \Parts(H)$ by the formula
    \[ \r M(F)^\kappa_\lambda = \sum_{\substack{\ell^\num = c(\kappa)^\num \\ \ell \in C(\lambda)}} \prod_i M(F)^{c(\kappa)_i^\dec}_{\ell_i^\dec}. \]
    This is a well-defined formula.
\end{proposition}
\begin{proof}
    There are two potential well-definition issues.
    First, the product $\prod_i M(F)^{c(\kappa)_i^\dec}_{\ell_i^\dec}$ with the domain of summation omitted is only defined when $\ell$ and $c(\kappa)$ have the same length (and thus the same set indexing their parts), but this is the case when $\ell^\num = c(\kappa)^\num$.
    Then, one must verify that this quantity is independent of the choice of composition $c(\kappa)$, but this follows from the fact that any two such compositions differ by a permutation in $\Sigma_{|\kappa|}$ which amounts to reordering the factors in $\prod_i M(F)^{c(\kappa)_i^\dec}_{\ell_i^\dec}$ for an $\ell$ differing by the same permutation.
\end{proof}

\begin{theorem} \label{thm:parksform}
    The map $\r F \colon \Parks(H,\Q) \to \Parks(G,\Q)$ is given by
    \[ \r F(f)(\kappa) = \sum_{\lambda \in \Parts(H)} \r M(F)^\kappa_\lambda f(\lambda) = \sum_{\substack{\lambda^\num = \kappa^\num \\ \lambda \in \Parts(H)}} \r M(F)^\kappa_\lambda f(\lambda). \]
    Equivalently, 
    \[ \r F(f)(\kappa) = \sum_{\substack{\ell^\num = c(\kappa)^\num \\ \ell \in C(\Conj(G)) }} f(u\ell) \prod_i M(F)^{c(\kappa)_i^\dec}_{\ell_i^\dec}. \]
\end{theorem}
\begin{proof}
By \ref{prop:powcorep}, to verify the correctness of this formula, it suffices to check that it is compatible with the power operations and the transfer product.
First, we check compatibility with the power operations.
For $f \in \Marks(H,\Q)$ and $\kappa \in \Parts(G,n)$, we compute that
\begin{align*}
    \P_n(F(f))(\kappa) 
    &= \prod_{\substack{[K] \in \Conj(G) \\ m \in \N_+}} F(f)([K])^{\kappa_{[K],m}}
    = \prod_i F(f)(c(\kappa)_i^\dec) 
    \\&= \prod_i \sum_{\ell_i \in \Conj(H)} M(F)^{c(\kappa)_i^\dec}_{\ell_i} f(\ell_i)
    \\&= \sum_{\substack{\ell^\num = c(\kappa)^\num \\ \ell \in C(\Conj(H))}} \prod_i M(F)^{c(\kappa)_i^\dec}_{\ell_i^\dec} f(\ell_i^\dec)
    \\&= \sum_{\substack{\ell^\num = c(\kappa)^\num \\ \ell \in C(\Conj(H))}} \P_n(f)(u\ell) \prod_i M(F)^{c(\kappa)_i^\dec}_{\ell_i^\dec} 
    \\&= \sum_{\lambda \in \Parts(H)} \P_n(f)(\lambda) \sum_{\substack{\ell^\num = c(\kappa)^\num \\ \ell \in C(\lambda)}} \prod_i M(F)^{c(\kappa)_i^\dec}_{\ell_i^\dec}
    \\&= \sum_{\lambda \in \Parts(H)} \r M(F)^\kappa_\lambda \P_n(f)(\lambda) = \r F(\P_n(f))(\kappa).
\end{align*}
The nontrivial equality on the first line translates between partitions and composition as in \eqref{eq:independent}.
The equality between the second and third lines follows from the fact that a choice $\ell_i \in \Conj(H)$ for each $i$ which indexes $c(\kappa)$ is the same data as choosing a decoration from $\Conj(H)$ for each of the parts of $c(\kappa)^\num$.
The equality between the fourth and fifth lines follows from the fact that summing over compositions is the same as first summing over partitions and then the compositions of those partitions.
The other equalities are applying definitions.

Next, we check compatibility with the transfer product.
For $f,g \in \Parks(H,\Q)$ and $\mu \in \Parts(G)$, we compute that
\begin{align*}
    (\r F(f)\star \r F(g))(\mu)
    &= \sum_{\substack{I \cupdot J = [|\mu|]}} \r F(f)(uc(\mu)|_I) \r F(g)(uc(\mu)|_J)
    \\&= \sum_{\substack{I \cupdot J = [|\mu|]}} \sum_{\substack{k^\num = c(\mu)^\num|_I \\ \ell^\num = c(\mu)^\num|_J \\ k \in C(I,\Conj(H)) \\ \ell \in C(J,\Conj(H))}} f(uk)g(u\ell) \prod_{i \in I} M(F)^{c(\mu)_i^\dec}_{k^\dec} \prod_{j \in J} M(F)^{c(\mu)_j^\dec}_{\lambda_j^\dec}
    \\&= \sum_{\substack{I \cupdot J = [|\mu|]}} \sum_{\substack{\ell^\num = c(\mu)^\num \\ \ell \in C(\Conj(H))}} f(u\ell|_I)g(u\ell|_J) \prod_i M(F)^{c(\mu)_i^\dec}_{\ell_i^\dec}
    \\&= \sum_{\substack{\ell^\num = c(\mu)^\num \\ \ell \in C(\Conj(H))}} \sum_{\substack{I \cupdot J = [|u\ell|]}} f(u\ell|_I)g(u\ell|_J) \prod_i M(F)^{c(\mu)_i^\dec}_{\ell_i^\dec}
    \\&= \sum_{\substack{\ell^\num = c(\mu)^\num \\ \ell \in C(\Conj(H))}} (f \star g)(u\ell) \prod_i M(F)^{c(\mu)_i^\dec}_{\ell_i^\dec} 
    = \r F(f \star g)(\mu).
\end{align*}
The equality of the first line comes from \ref{prop:compprod}.
The equality between the second and third line follows from the bijection $C(I,\Conj(H)) \times C(J,\Conj(H)) \cong C([|\mu|], \Conj(H))$.
The other equalities follow from manipulating definitions and summation order.
\end{proof}

\begin{corollary}
    Let $\lambda \in \Parts(H)$ and let $\1_\lambda \in \Parks(H,\Q)$ denote the indicator function of $\lambda$.
    Then,
    \[ \r F(\1_\lambda) = \sum_{\kappa \in \Parks(G)} \r M(F)^\kappa_\lambda \1_\kappa \]
    so that
    \[ \r F(\1_\lambda)(\kappa) = \r M(F)^\kappa_\lambda = \sum_{\substack{\ell^\num = c(\kappa)^\num \\ \ell \in C(\lambda)}} \prod_i M(F)^{c(\kappa)_i^\dec}_{\ell_i^\dec}. \]
\end{corollary}

\begin{corollary}
    If the matrix components $M(F)^{[K]}_{[L]}$ are elements of a commutative ring $R$, then the formula of \ref{thm:parksform} provides a commutative graded ring map
    \[ \r F \colon \Parks(H,R) \to \Parks(G,R) \]
    which fits into a commutative square
    \[\begin{tikzcd}
        \AA(H) \ar[r, "\r F"] \ar[d, "\chi"'] & \AA(G) \ar[d, "\chi"] \\
        \Parks(H,R) \ar[r, "\r F"] & \Parks(G,R)
    \end{tikzcd}\]
    witnessing compatibility with the character map.
\end{corollary}
In particular, the results of \ref{sec:parksgentr}, \ref{sec:parksspetr}, and \ref{sec:parksres} apply to an arbitrary coefficient ring $R$ so long as the coefficients of $M(F)$ are elements of $R$.

\section{A general formula for transfers on parks} \label{sec:parksgentr}

In this section, we analyze the specific case of \ref{thm:parksform} when the map $F \colon A(H) \to A(G)$ is a transfer along a (not necessarily injective) group homomorphism.
We only rely on one fact about transfers along group homomorphisms: that certain values $M(F)^{[K]}_{[L]}$ are zero.
Thus, we start by proving a generalization of the result for transfers along group homomorphisms, and then conclude the result about group homomorphisms as a corollary. We continue to use the notation of the previous section. 

\begin{proposition} \label{prop:parksgentr}
    Let $F \colon A(H) \to A(G)$ be a map of abelian groups.
    Suppose that for fixed $[L]$, there is at most one $[K]$ such that $M(F)^{[K]}_{[L]}$ is nonzero.
    Equivalently, suppose there is a function $\psi \colon \Conj(H) \to \Conj(G)$ with the property that $M(F)^{[K]}_{[L]}$ is nonzero only if $\psi [L] = [K]$.
    Write $M(F)_{[L]}$ for $M(F)^{\psi[L]}_{[L]}$ so that
    \[ F(f)([K]) = \sum_{[L] \in \psi^{-1}[K]} M(F)_{[L]} f([L]) \]
    for $f \in \Marks(H,\Q)$ and $[K] \in \Conj(G)$.
    Define coefficients 
    \[ \r M(F)_\lambda = \frac{(\psi_*\lambda)!}{\lambda!}\prod_{[L],m} (M(F)_{[L]})^{\lambda_{[L],m}} \]
    for $\lambda \in \Parts(H)$.
    Then, $\r F \colon \Parks(H,\Q) \to \Parks(G,\Q)$ has the following formula:
    \[ \r F(f)(\kappa) = \sum_{\lambda \in \psi_*^{-1}\kappa} \r M(F)_\lambda f(\lambda) \]
    for $f \in \Parks(H,\Q)$ and $\kappa \in \Parts(G)$, where $\psi_*^{-1}$ denotes the preimage under $\psi_* \colon \Parts(H) \to \Parts(G)$. 
    Equivalently,
    \[ \r F(f)(\kappa) = \sum_{\lambda \in \psi_*^{-1}\kappa} \frac{\kappa!}{\lambda!} f(\lambda) \prod_{[L],m} (M(F)_{[L]})^{\lambda_{[L],m}}. \] 
    Note that the quantities $\kappa!/\lambda! = \psi_*\lambda!/\lambda!$ are integral.
\end{proposition}
\begin{proof}
    To prove this claim, it suffices to study the coefficients 
    \[ \r M(F)^\kappa_\lambda = \sum_{\substack{\ell^\num = c(\kappa)^\num \\ \ell \in C(\lambda)}} \prod_i M(F)^{c(\kappa)_i^\dec}_{\ell_i^\dec} \] 
    used in \ref{thm:parksform}.
    First, note that if $\psi_*\ell \neq c(\kappa)$, the product $\prod_i M(F)^{c(\kappa)_i^\dec}_{\ell_i^\dec}$ will be zero since at least one of the values $M(F)^{c(\kappa)_i^\dec}_{\ell_i^\dec}$ will be zero.
    Further, in the case that $\psi_*\ell = c(\kappa)$, this product is simply $\prod_i M(F)_{\ell_i^\dec}$.
    Thus,
    \[ \r M(F)^\kappa_\lambda = \sum_{\substack{\psi_*\ell = c(\kappa) \\ \ell \in C(\lambda)}} \prod_i M(F)_{\ell_i^\dec} = \left| \{ \ell \in C(\lambda) \colon \psi_*\ell = c(\kappa) \} \right| \prod_{[L],m} (M(F)_{[L]})^{\lambda_{[L],m}}. \] 
    Thus, it suffices to enumerate those $\ell \in C(\lambda)$ such that $\psi_*\ell = c(\kappa)$.
    There exists such an $\ell$ if and only if $\psi_*\lambda = \kappa$.
    In such a case, choose $c(\lambda)$ such that $\psi_*c(\lambda) = c(\kappa)$.
    Any other choice of $\ell$ differs from $c(\lambda)$ by permutation in $\Sigma_{|\lambda|}$ which fixes $c(\kappa)$ i.e.\ a permutation in $\Sigma_{c(\kappa)}$.
    Two such permutations give the same choice of $\ell$ whenever $\sigma \in \Sigma_{c(\lambda)}$.
    Thus,
    \begin{equation} \label{eq:compcounting}
        \left| \{ \ell \in C(\lambda) \colon \psi_*\ell = c(\kappa) \} \right| = [\Sigma_{c(\kappa)} \colon \Sigma_{c(\lambda)}] = \frac{\kappa!}{\lambda!}.
    \end{equation}
    The result now follows.
\end{proof}

\begin{corollary}
    On indicator functions,
    \[ \r F(\1_\lambda) = \r M(F)_\lambda \1_{\psi_* \lambda} = \frac{(\psi_*\lambda)!}{\lambda!} \1_{\psi_* \lambda} \prod_{[L],m} (M(F)_{[L]})^{\lambda_{[L],m}} \]
    for $\lambda \in \Parts(H)$.
\end{corollary}

Let $\phi\colon H \to G$ be a homomorphism of finite groups. The map $\Tr_\phi \colon A(H) \to A(G)$ is of the form required for \ref{prop:parksgentr}.
The map $\psi$ of the proposition is the induced map $\phi \colon \Conj(H) \to \Conj(G)$.
That is, the transfer map on marks is determined by values $M(\Tr_\phi)_{[L]} \in \Q$ for $[L] \in \Conj(H)$ such that
\[ \Tr_\phi(f)([K]) = \sum_{[L] \in \phi^{-1}([K])} M(\Tr_\phi)_{[L]} f([L]), \]
where $\phi^{-1}$ denotes the preimage under $\phi \colon \Conj(H) \to \Conj(G)$.
For arbitrary $\phi$, it is difficult to give an elegant formula for the coefficients $M(\Tr_\phi)_{[L]}$, or even for the transfer map on marks.
A formula for the coefficients $M(\Tr_\phi)_{[L]}$ in terms of M\"oebius inversion can be extracted from the results of Bouc in \cite[Chapter 5]{Bouc2010}.

\begin{corollary} \label{cor:parkstr}
    Let $\phi \colon H \to G$ be a homomorphism of finite groups.
    Then, the transfer map on parks $\Tr_\phi \colon \Parks(H,\Q) \to \Parks(G,\Q)$
    has the following formula:
    \[ \Tr_\phi(f)(\kappa) = \sum_{\lambda \in \phi_*^{-1}\kappa} \r M(\Tr_\phi)_\lambda f(\lambda) \]
    for $f \in \Parks(H,\Q)$ and $\kappa \in \Parts(G)$. 
    On indicator functions,
    \[ \Tr_\phi(\1_\lambda) = \r M(\Tr_\phi)_\lambda \1_{\phi_* \lambda} \]
    for $\lambda \in \Parts(H)$.
\end{corollary}
\begin{proof}
    This follows immediately from \ref{prop:parksgentr} using $F=\Tr_\phi$ and $\psi = \phi \colon \Conj(H) \to \Conj(G)$.
\end{proof}

\section{Special cases of the transfer formula} \label{sec:parksspetr}

Although it is difficult to give an elegant formula for the transfer map on marks for an arbitrary homomorphism $\phi \colon H \to G$, some homomorphisms have a convenient transfer formula.
In this section, we seek to provide alternative formulas for \ref{cor:parkstr} in such cases.
In particular, we consider subgroup inclusions and the map $H \to e$.

We begin by discussing subgroup inclusions.
Recall the classical formula for the transfer on marks along a subgroup inclusion
\begin{equation} \label{eq:marksinctr}
    \Tr_{H \subseteq G}(f)([K]) = \sum_{\substack{gH \in G/H \\ K^g \subseteq H}} f([K^g]).
\end{equation}
In particular, this involves choosing a representative for the conjugacy class and the conjugation action on subgroups.
We seek to reproduce these ideas on compositions.

Note that if $G$ acts on $X$ (say, on the right), then $G^{I}$ acts on $C(I,X)$ by acting on decorations.
Let $q \colon X \to X/G$ denote the orbit map.
Then, $q_* \colon C(I,X) \to C(I,X/G)$ induces to an isomorphism $C(I,X)/G^{I} \to C(I,X/G)$.
Of particular utility is the conjugation action of $G$ on $\Sub(G)$, where $\Sub(G)$ denotes the set of subgroups of $G$. Recall the map $u$ of \ref{sec:decoratedcompositions} that associates to a decorated composition its underlying decorated partition.

\begin{proposition} \label{prop:parksinctr}
    Let $\phi$ be a subgroup inclusion $H \subseteq G$.
    Let $f \in \Parks(H)$ and $\kappa \in \Parts(G)$.
    Choose a composition $k \in q_*^{-1}C(\kappa) \subseteq C(\Sub(G))$ and let $c(\kappa)=q_*k$.
    The composition $k$ amounts to a choice of ordering of the parts of $\kappa$, along with a choice of representative for each of the conjugacy classes that decorate the parts.   

    Then, $\Tr_{H \subseteq G} \colon \Parks(H,\Q) \to \Parks(G,\Q)$ is given by
    \[ \Tr_{H \subseteq G}(f)(\kappa) = \sum_{\substack{gH^{\times |\kappa|} \in G^{\times |\kappa|}/H^{\times |\kappa|} \\ k^g \in C(|\kappa|, \Sub(H))}} f(uq_*k^g), \]
    where $C(|\kappa|, \Sub(H)) \subseteq C(|\kappa|, \Sub(G))$.
\end{proposition}
\begin{proof}
    The coefficients $M(\Tr_{H \subseteq G})_{[L]}$ used in \ref{prop:parksgentr} are given by
    \begin{equation} \label{eq:inccoefs}
        M(\Tr_{H \subseteq G})_{[L]} = [N_G(L) \colon N_H(L)].
    \end{equation}
    We will prove \ref{prop:parksinctr} using \eqref{eq:inccoefs} and \ref{cor:parkstr}.
    Notice that the formula of \ref{prop:parksinctr} on $\Parks(H,1,\Q) \cong \Marks(H,\Q)$ is equivalent to the classical formula \eqref{eq:marksinctr}.
    Thus, our proof of \ref{prop:parksinctr} will additionally verify the correctness of \eqref{eq:inccoefs}.

    Let $=_G$ denote the equivalence relation of two subgroups being $G$-conjugate and $=_H$ denote the equivalence relation of two subgroups being $H$-conjugate.
    Choose subgroups $L \subseteq H$ and $K \subseteq G$ and suppose $K =_G L$.
    Choose $a \in G$ witnessing that $K^a = L$.
    Then,
    \[ \{ g \in G/H \colon K^g =_H L \} = a \{ gH \colon L^g =_H L \} = aN_G(L)H/H \cong N_G(L)/N_H(L) \]
    so that
    \begin{equation} \label{eq:inccounting}
        \left| \{ gH \colon K^g \in [L] \} \right| = \left| \{ gH \colon K^g =_H L \} \right| = \frac{|N_G(L)|}{|N_H(L)|} = M(\Tr_{H \subseteq G})_{[L]}.
    \end{equation}
    Notice that $[L] \in \phi_*^{-1}([K])$ if and only if $L =_G K$.
    Thus, \eqref{eq:inccounting} holds when $\phi_*[L] = [K]$.

    Now, we compute that
    \begin{align*}
        \sum_{\substack{gH^{\times |\kappa|} \in G^{\times |\kappa|}/H^{\times |\kappa|} \\ k^g \in C(\Sub(H))}} f(uq_*k^g)
        &= \sum_{\substack{gH^{\times |\kappa|} \in G^{\times |\kappa|}/H^{\times |\kappa|} \\ \ell \in C(\Conj(H)) \\ k^g \in C(\Sub(H)) \\ q_*k^g = \ell}} f(u\ell)
        \\ &= \sum_{\substack{\ell \in C(\Conj(H))}} f(u\ell) \left|\left\{ gH^{\times |\kappa|} \in G^{\times |\kappa|}/H^{\times |\kappa|} \colon \substack{k^g \in C(\Sub(H)) \\ q_*k^g = \ell} \right\}\right| 
        \\ &= \sum_{\substack{\phi_*\ell = q_*k \\ \ell \in C(\Conj(H))}} f(u\ell) \left|\left\{ gH^{\times |\kappa|} \in G^{\times |\kappa|}/H^{\times |\kappa|} \colon \substack{(k_i^\dec)^{g_i} \subseteq H \\ (k_i^\dec)^{g_i} \in \ell^i_\dec}\ \forall i \right\}\right| 
        \\ &= \sum_{\substack{\phi_*\ell = q_*k \\ \ell \in C(\Conj(H))}} f(u\ell) \prod_i \left|\left\{ gH \in G/H \colon (k_i^\dec)^g \in \ell^i_\dec \right\}\right| 
        \\ &= \sum_{\substack{\phi_*\ell = c(\kappa) \\ \ell \in C(\Conj(H))}} f(u\ell) \prod_i M(\Tr_{H \subseteq G})_{\ell_i^\dec} 
        \\ &= \sum_{\lambda \in \phi_*^{-1}\kappa} f(\lambda) \sum_{\substack{\phi_*\ell = c(\kappa) \\ \ell \in C(\lambda)}} \prod_{[L],m} (M(\Tr_{H \subseteq G})_{[L]})^{\lambda_{[L],m}} 
        \\ &= \sum_{\lambda \in \phi_*^{-1}\kappa} \frac{\kappa!}{\lambda!} f(\lambda) \prod_{[L],m} (M(\Tr_{H \subseteq G})_{[L]})^{\lambda_{[L],m}} 
        = \Tr_{H \subseteq G}(f)(\kappa).
    \end{align*}
    The first equality follows by reparameterizing the sum according to the value of $q_*\kappa^g$.
    The second equality follows by recognizing the the summand is independent of most of the sum's quantifiers.
    The third equality follows by restricting the sum to only those compositions whose corresponding counting problem gives a nonzero result and by replacing statements about compositions with index-wise statements.
    The fourth equality follows by replacing the index-wise counting problem with a product of counting problems.
    The fifth equality follows from \eqref{eq:inccounting}.
    The sixth equality follows from \eqref{eq:independent} and reparameterizing the sum according to partitions.
    The seventh equality follows from \eqref{eq:compcounting}.
\end{proof}

Next, we consider transfers to the trivial group.
Recall the Burnside orbit counting lemma which gives a formula for the transfer $A(G) \to A(e)$ on marks:
\begin{equation} \label{eq:markstrivtr}
    \Tr_{G \to e}(f)([e]) = \frac{1}{|G|}\sum_{g \in G} f([\ag g]),
\end{equation}
for $f \in \Marks(G,\Q)$.
In particular, this formula involves taking conjugacy classes of cyclic subgroups generated by $g \in G$.
We will need a version of this for compositions.

Let $s \colon G \to \Conj(G) \colon g \mapsto [\ag g]$ be the function that maps an element to the conjugacy class of its cyclic subgroup.
Then, we may consider $s_* \colon C(G) \to C(\Conj(G))$ which replaces decorations by elements of $G$ with the conjugacy class of the cyclic subgroup they generate.
\begin{proposition} \label{prop:parkstrivtr}
    The transfer $\Tr_{G \to e} \colon \Parks(G,\Q) \to \Parks(e,\Q)$ is given by
    \[ \Tr_{G \to e}(f)(\lambda) = \frac{1}{|G|^{|\lambda|}} \sum_{\substack{c^\num = c(\lambda) \\ c \in C(G)}} f(us_*c) \]
    for $f \in \Parks(G,\Q)$ and an integer partition $\lambda \in \Parts(e)$.
\end{proposition}
    Notice that there is a bijection between the set of $c \in C(G)$ such that $c^\num = c(\lambda)$ and $G^{\times |\lambda|}$, so this is really a generalization of Burnside's orbit counting lemma.
\begin{proof}
    The coefficients $M(\Tr_{G \to e})_{[K]}$ used in \ref{prop:parksgentr} are given 
    \begin{equation} \label{eq:trivcoefs}
        M(\Tr_{G \to e})_{[K]} = \begin{cases} | (\Z/|K|)^\times | / |N_G(K)| & K \text{ cyclic} \\ 0 & \text{otherwise.} \end{cases}
    \end{equation}
    We will prove \ref{prop:parkstrivtr} using \eqref{eq:trivcoefs} and \ref{cor:parkstr}.
    Notice the the formula of \ref{prop:parkstrivtr} on $\Parks(G,1,\Q) \cong \Marks(G,\Q)$ is equivalent to Burnside orbit counting lemma \eqref{eq:markstrivtr}.
    Thus, our proof of \ref{prop:parkstrivtr} will additionally verify the correctness of \eqref{eq:trivcoefs}.

    Suppose that $H \subseteq G$ is a cyclic subgroup.
    Then,
    \begin{align*}
        \{ g \in G \colon [\ag g] = [H] \} 
        &\cong \coprod_{K =_G H} \{ g \colon \ag g = K \}
        \\ &\cong \{K \colon K =_G H\} \times (\Z/|H|)^\times \cong G/N_G(H) \times (\Z/|H|)^\times.
    \end{align*}
    Otherwise, this set is empty.
    Thus,
    \begin{equation} \label{eq:trivcounting}
        |\{ g \colon [\ag g] = [H] \}| = |G| M(\Tr_{G \to e})_{[H]}.
    \end{equation}
    Let $\phi$ denote the map $G \to e$.
    Notice that $\phi_*\kappa = \kappa^\num$ for $\kappa \in \Parts(G)$.
    Then, we compute that
    \begin{align*}
        \frac{1}{|G|^{|\lambda|}} \sum_{\substack{c^\num = c(\lambda) \\ c \in C(G)}} f(us_*c)
        &= \frac{1}{|G|^{|\lambda|}} \sum_{\substack{k^\num = c(\lambda),\, s_*c = k \\ k \in C(\Conj(G)),\, c \in C(G)}} f(uk)
        \\&= \frac{1}{|G|^{|\lambda|}} \sum_{\substack{k^\num = c(\lambda) \\ k \in C(\Conj(G))}} f(uk) |\{ c \in C(G) \colon s_*c = k \}| 
        \\&= \frac{1}{|G|^{|\lambda|}} \sum_{\substack{k^\num = c(\lambda) \\ k \in C(\Conj(G))}} f(uk) \prod_i |\{ g \colon [\ag g] =k_i^\dec \}|
        \\&= \frac{1}{|G|^{|\lambda|}} \sum_{\substack{k^\num = c(\lambda) \\ k \in C(\Conj(G))}} f(uk) \prod_i |G|M(\Tr_{G \to e})_{k_i^\dec} 
        \\&= \sum_{\substack{k^\num = c(\lambda) \\ k \in C(\Conj(G))}} f(uk) \prod_i M(\Tr_{G \to e})_{k_i^\dec} 
        \\&= \sum_{\substack{\kappa^\num = \lambda \\ \kappa \in \Parts(G)}} f(\kappa) \sum_{\substack{k^\num = c(\lambda) \\ k \in C(\kappa)}} \prod_{[K],m} (M(\Tr_{G \to e})_{[K]})^{\kappa_{[K],m}} 
        \\&= \sum_{\substack{\kappa^\num = \lambda \\ \kappa \in \Parts(G)}} \frac{\lambda!}{\kappa!} f(\kappa) \prod_{[K],m} (M(\Tr_{G \to e})_{[K]})^{\kappa_{[K],m}} 
        \\&= \Tr_{G \to e}(f)(\lambda).
    \end{align*}
    The first equality follows from reparameterizing the sum according to the value of $s_*k$.
    The second equality follows by recognizing the the summand is independent of most of the sum's quantifiers.
    The third equality follows by replacing the counting problem by the product over each index of a counting problem for that index.
    The fourth equality follows from \eqref{eq:trivcounting}.
    The sixth equality follows from \eqref{eq:independent} and reparameterizing the sum according to partitions.
    The seventh equality follows from \eqref{eq:compcounting}.
\end{proof}

The proposition above provides a formula for the transfer $\AA(G,n) \to A(\Sigma_n)$ on the level of parks and marks. 
That is, on the subring $\AA(G,n) \subseteq A(G \wr \Sigma_n)$, the transfer $A(G \wr \Sigma_n) \to A(\Sigma_n)$ admits a Burnside-lemma-like description. 

It is worth noting that such a formula should have been expected. The commutativity of the diagram
\[
\xymatrix{\AA(G,n) \ar[r] \ar[d]^-{\Tr} & A(G \wr \Sigma_n) \ar[r] \ar[d]^-{\Tr} & RU(G \wr \Sigma_n) \ar[d]^-{\Tr} \\ \AA(e,n) \ar[r] & A(\Sigma_n) \ar[r] & RU(\Sigma_n)}
\]
and the fact that the composite along the bottom is an isomorphism implies that the formula for the transfer $\Parks(G,n,\Q) \to \Parks(e,n,\Q)$ can be given in terms of the formula for the transfer in representation theory. A consequence of this observation is that, since $A(C)$ embeds in $RU(C)$ for any cyclic group $C$, there is always a formula analogous to Burnside's orbit counting lemma for the transfer along a surjection from a finite group onto a cyclic group.

\section{Formulas for restrictions on parks} \label{sec:parksres}

In this section, we analyze \ref{thm:parksform} in the case where the map $F \colon A(H) \to A(G)$ is a restriction.
In fact, we consider a slight generalization of this situation, which includes the Frobenius--Wielandt map as an example.

\begin{proposition} \label{prop:parksgenres}
    Let $F \colon A(H) \to A(G)$ be a map of abelian groups.
    Suppose that $F$ is given on marks by restriction along a function.
    Equivalently, suppose there is a map $\psi \colon \Conj(G) \to \Conj(H)$ with the property that
    \[ F(f)([K]) = f(\psi[K]) \]
    for $f \in \Marks(H,\Q)$.
    Then, $\r F \colon \Parks(H,\Q) \to \Parks(G,\Q)$ has the following formula:
    \[ \r F(f)(\kappa) = f(\psi_*\kappa) \]
    for $f \in \Parks(H,\Q)$ and $\kappa \in \Parts(G)$, where $\psi_*$ denotes the pushforward $\Parts(G) \to \Parts(H)$.
    On indicator functions,
    \[ \r F(\1_\lambda) = \sum_{\kappa \in \psi_*^{-1}\lambda} \1_\kappa \]
    for $\lambda \in \Parts(H)$, where $\psi_*^{-1}$ denotes the preimage under $\psi_*$.
\end{proposition}
\begin{proof}
    This is a simple consequence of \ref{thm:parksform}.
    First, notice that the matrix components $M(F)$ are given by
    \[ M(F)^{[K]}_{[L]} = \begin{cases} 1,& \psi[K] = [L] \\ 0, &\text{otherwise.} \end{cases} \]
    Thus, we compute that
    \begin{align*}
        \r M(F)^\kappa_\lambda 
        &= \sum_{\substack{\ell^\num = c(\kappa)^\num \\ \ell \in C(\lambda)}} \prod_i M(F)^{c(\kappa)_i^\dec}_{\ell_i^\dec}
        = \sum_{\substack{\ell = \psi_*c(\kappa) \\ \ell \in C(\lambda)}} 1
        = \begin{cases} 1,& \psi_*\kappa = \lambda \\ 0, &\text{otherwise.} \end{cases}
    \end{align*}
    This completes the proof.
\end{proof}

Let $\phi \colon G \to H$ be a homomorphism of finite groups.
The restriction map $\Res_\phi$ has a very simple formula on marks:
\[ \Res_\phi(f)([K]) = f(\phi[K]) = f([\phi(K)]) \]
for $f \in \Marks(H,\Q)$, where $\phi \colon \Conj(G) \to \Conj(H)$ is the induced map on conjugacy classes of subgroups.
Thus, we may specialize the previous proposition to the case of restriction along a group homomorphism.
\begin{corollary}
    Let $\phi \colon G \to H$ be a homomorphism of finite groups.
    Then, the restriction map on parks $\Res_\phi \colon \Parks(H,\Q) \to \Parks(G,\Q)$
    has the following formula:
    \[ \Res_\phi(f)(\kappa) = f(\phi_*\kappa) \]
    for $f \in \Parks(H,\Q)$ and $\kappa \in \Parts(G)$.
    On indicator functions,
    \[ \Res_\phi(\1_\lambda) = \sum_{\kappa \in \phi_*^{-1}\lambda} \1_\kappa \]
    for $\lambda \in \Parts(H)$.
\end{corollary}
\begin{proof}
    This follows immediately from \ref{prop:parksgenres} using $F=\Res_\phi$ and $\psi=\phi\colon \Conj(G) \to \Conj(H)$.
\end{proof}

\section{The Frobenius--Wielandt map and norms} \label{sec:weird}

\ref{prop:parksgenres} is relevant to more than restriction along group homomorphisms.
In particular, we may apply it to the Frobenius--Wielandt map. In this section, we extend the Frobenius--Wielandt map to the Burnside ring of the wreath product compatibly with the total power operation and describe the relation to work of Romero in \cite{Romero2020}.

Let $C_m$ denote the cyclic group of size $m$. 
When $d \divides m$, we will identify $C_d$ with the unique order $d$ subgroup of $C_m$. Further, for a finite group $G$, let $C_G$ denote the cyclic group with the same order as $G$ so that $C_G = C_{|G|}$.

Recall the Frobenius--Wielandt map $w \colon A(C_G) \to A(G)$.
The map $w$ is given on marks by precomposition with $\Cyc \colon \Conj(G) \to \Conj(C_G)$ sending $[H]$ to the conjugacy class of $C_H \subseteq C_G$. 
Thus, $w$ extends to a commutative graded ring homomorphism $\wA\colon \AA(C_G) \to \AA(G)$, and the induced map $\wA \colon \Parks(C_G,\Q) \to \Parks(G,\Q)$ is given by precomposition with $\Cyc_* \colon \Parts(G) \to \Parts(C_G)$. We have the following corollary:

\begin{corollary}
There is a commutative diagram 
\[
\xymatrix{A(C_G) \ar[r]^-{w} \ar[d]_-{\P_n} & A(G) \ar[d]^-{\P_n} \\ \AA(C_G,n) \ar[r]^-{\wA} & \AA(G,n).}
\]
\end{corollary}

We may use $\wA \colon \AA(C_G,n) \to \AA(G,n)$ to define a map
\[
w_n \colon A(C_G \wr \Sigma_n) \to A(G \wr \Sigma_n)
\]
that is compatible with $w \colon A(C_G) \to A(G)$ through the total power operation $\P_n$. We define $w_n$ to be the composite
\[
w_n \colon A(C_G \wr \Sigma_n) \xrightarrow{r} \AA(C_G,n) \xrightarrow{\wA} \AA(G,n) \hookrightarrow A(G \wr \Sigma_n),
\]
where $r$ is the map of \ref{def:r}.

\begin{corollary} \label{cor:FWandpower}
There is a commutative diagram 
\[
\xymatrix{A(C_G) \ar[r]^-{w} \ar[d]_-{\P_n} & A(G) \ar[d]^-{\P_n} \\ A(C_G \wr \Sigma_n) \ar[r]^-{w_n} & A(G \wr \Sigma_n).}
\]
\end{corollary}


Recall that, for $H \subseteq G$, the norm (also called tensor induction) from $H$ to $G$ is the multiplicative map
\[
\Nm_{H}^{G} \colon A(H) \to A(G)
\]
induced by the map sending an $H$-set $X$ to the $G$-set $\Fun_H(G,X)$ of $H$-equivariant maps from $G$ to $X$. It is compatible with the function $\Nm_{H}^{G} \colon \Marks(H) \to \Marks(G)$ given by
\[
\Nm_{H}^{G}(f)([K]) = \prod_{KgH \in K \backslash G /H} f([K^g \cap H]).
\]
Further, the norm can be constructed from the total power operation. 
Let $n=[G: H]$.
The monomial representation (or relative Cayley map) is an injective homomorphism $G \to H \wr \Sigma_n$, which is well-defined up to conjugacy (see \cite[Section 5.2]{Evens} for a complete description). 
The norm is the composite
\[
A(H) \xrightarrow{\P_n} A(H \wr \Sigma_n) \xrightarrow{\Res_{G}^{H \wr \Sigma_n}} A(G),
\]
where the restriction is along the monomial representation.

In \cite{Romero2020}, Romero considers the interaction between the Frobenius--Wielandt map and norms for Burnside rings. Let $H \subseteq G$ be a subgroup. The subgroup $H$ is said to satisfy the ``gcd property'' if for all subgroups $K \subseteq G$,
\[
|K \cap H| = \gcd(|K|,|H|).
\]
Subgroups satisfying the gcd property are necessarily normal.
Romero proves that the the norm from $H$ to $G$ commutes with the Frobenius--Wielandt map if and only if $H$ satisfies the gcd property.

However, \ref{cor:FWandpower} explains that the Frobenius--Wielandt map can always be extended to the target of the total power operation. 
This suggests that the gcd property is really a requirement on the restriction map $\Res_{G}^{H \wr \Sigma_n}$ along the monomial representation, and the next two results will prove that this is in fact the case.

Since the total power operation factors through $\AA(H,n)$, it suffices to understand the composite
\[
\AA(H,n) \hookrightarrow A(H \wr \Sigma_n) \xrightarrow{\Res_{G}^{H\wr \Sigma_n}} A(G).
\]
We will call this composite $R_G^H$.

\begin{proposition} \label{prop:evansform}
    Let $H \subseteq G$ and let $n = [G:H]$.
    The restriction map $R_G^H$ described above is compatible with the map $R_G^H \colon \Parks(H, n) \to \Marks(G)$ given by the formula
    \[ R_G^H(f)([K]) = f\left(\sum_{KaH \in K \backslash G /H} ([K^a \cap H], |KaH|/|H|) \right) \]
    for $f \in \Parks(H,n)$ and $[K] \in \Conj(G)$.
\end{proposition}
\begin{proof}
    Let $\phi \colon G \to H \wr \Sigma_n$ denote the monomial representation.
    The map $R_G^H \colon \Parks(H,n) \to \Marks(G)$ is the composite
    \[ R_G^H \colon \Parks(H,n) \xrightarrow{\beta^*} \Marks(H \wr \Sigma_n) \xrightarrow{\Res_\phi} \Marks(G), \]
    and the map $\Res_\phi$ on marks is given by precomposition with $\phi \colon \Conj(G) \to \Conj(H \wr \Sigma_n)$.
    Thus, to understand $R_G^H$, we only need to understand the map $\beta\phi \colon \Conj(G) \to \Parts(H,n)$.
    We will prove the proposition by showing that
    \[ \beta\phi[K] = \sum_{KaH \in K \under G/H} ([K^a \cap H], |KaH|/|H|). \]
    
    For the sake of this proof, we will choose a particular monomial representation, among the many conjugate choices.
    Let $[n]$ be the set $G/H$ so that $\Sigma_n = \Sigma_{G/H}$ and $H^{\times n} = H^{G/H}$.
    Let $\phi_\Sigma \colon G \to \Sigma_n$ be the (Cayley) homomorphism defined by
    \[ \phi_\Sigma(g) \colon aH \mapsto gaH. \]
    Choose a section $s \colon G/H \to G$ of the quotient $G \to G/H$.
    We will write $s(a)$ for $s(aH)$ to reduce clutter.
    Define a function $\phi_H \colon G \to H^{\times n}$ by
    \[ \phi_H(g)_{aH} = s(a)^{-1} g s(g^{-1}a). \]
    Then, the monomial representation is the homomorphism 
    \[ \phi\colon G \to H \wr \Sigma_n \colon g \mapsto (\phi_H(g), \phi_\Sigma(g)). \]
    To see that this is a homomorphism, note that
    \begin{align*}
        \phi(g)\phi(g') 
        &= (\phi_H(g)\phi_\Sigma(g)\phi_H(g'), \phi_\Sigma(g)\phi_\Sigma(g'))
        = (\phi_H(gg'), \phi_\Sigma(gg'))
        = \phi(gg'),
    \end{align*}
    where the nontrivial equality follows from the fact that
    \[ (\phi_H(g)\phi_\Sigma(g)\phi_H(g'))_{aH} = s(a)^{-1} g s(g^{-1}a) s(g^{-1}a)^{-1} g' s(g'^{-1}g^{-1}a) = \phi(gg')_{aH}. \]
    
    By \ref{def:beta},
    \[ \beta\phi[K] = \sum_{\phi(K) i \in \phi(K)\under [n]} ([\phi(K)_{ii}], |\phi(K)i|). \]
    Now, we take advantage of the fact that $[n] = G/H$ and that the action of $K$ on $G/H$ is left multiplication.
    Then, we may identify the quotient $\phi(K) \under [n]$ with $K \under G/H$ and the quantity $|\phi(K)i|$ with $|KaH|/|H|$, the number of $H$-cosets in the $K$-orbit $KaH$, where $i \in [n]$ is identified with $aH \in G/H$.
    All that remains is to understand the subgroup $\phi(K)_{ii} = \phi(K)_{aH\,aH}$.

    By definition $\phi(K)_{aH\,aH} = \{ \phi_H(g)_{aH} \mid g \in G,\, gaH=aH \}$.
    In the situation that $gaH = aH$, it follows that $\phi_H(g)_{aH} = s(a)^{-1} g s(a)$.
    Further, $gaH=aH$ if and only if $a^{-1}ga \in H$ if and only if $s(a)^{-1} g s(a) \in H$.
    Thus, it follows that $\phi(K)_{aH\,aH} = K^{s(a)} \cap H$ and thus that $[\phi(K)_{aH\,aH}] = [K^a \cap H]$.

    Assembling these observations, we find that
    \[ \beta\phi[K] = \sum_{KaH \in K \under G/H} ([K^a \cap H], |KaH|/|H|). \qedhere \]
\end{proof}

\begin{proposition}
    Assume $H \subseteq G$ and let $n=[G:H]$. The diagram
    \begin{equation} \label{eq:tensor square}
    \begin{tikzcd}
        \AA(C_H,n) \ar[r, "R_{C_G}^{C_H}"] \ar[d, "\wA"] & A(C_G) \ar[d, "w"] \\
        \AA(H,n) \ar[r, "R_G^H"] & A(G)
    \end{tikzcd}
    \end{equation}
    commutes if and only if $H$ satisfies the gcd property.
\end{proposition}
\begin{proof}
    The square \eqref{eq:tensor square} commutes if and only if the same corresponding square on parks and marks commutes.
    Using \ref{prop:evansform}, all of the maps on parks and marks are given by precomposition with known functions, so \eqref{eq:tensor square} commutes if and only if
    \[\begin{tikzcd}
        \Conj(G) \ar[r, "\beta\phi"] \ar[d, "\Cyc"'] & \Parts(H,n) \ar[d, "\Cyc_*"] \\
        \Conj(C_G) \ar[r, "\beta\phi"] & \Parts(C_H,n)
    \end{tikzcd}\]
    commutes.
    For $[K] \in \Conj(G)$, the top composite is given by
    \begin{align*}
        \Cyc_* \beta\phi([K]) 
        &= \Cyc_*\left( \sum_{KaH \in K \under G/H} ([K^a \cap H], |KaH|/|H|) \right)
        \\ &= \sum_{KaH \in K \under G/H} ([C_{K^a \cap H}], |KaH|/|H|)
    \end{align*}
    while the bottom composite is given by
    \begin{align*}
        \beta\phi\Cyc([K])
        &= \beta\phi([C_K])
        = \sum_{C_K m C_H \in C_K \under C_G/C_H} ([{C_K}^m \cap C_H], |C_K m C_H|/|C_H|)
        \\&= \frac{|G|}{\lcm(|K|,|H|)} \left([C_{\gcd(|K|,|H|)}], \frac{\lcm(|K|,|H|)}{|H|}\right).
    \end{align*}
    Thus, \eqref{eq:tensor square} commutes if and only if
    \begin{equation} \label{eq:cyc equation}
        \sum_{KaH \in K \under G/H} ([C_{K^a \cap H}], |KaH|/|H|)
        = \frac{|G|}{\lcm(|K|,|H|)} \left([C_{\gcd(|K|,|H|)}], \frac{\lcm(|K|,|H|)}{|H|}\right)
    \end{equation}
    for each subgroup $K \subseteq G$.
    
    Suppose $H$ satisfies the gcd property. 
    Then, for each subgroup $K \subseteq G$ and $a \in G$, $|K^a \cap H| = \gcd(|K|,|H|)$ so that $C_{K^a \cap H} = C_{\gcd(|K|,|H|)}$.
    Further, $H$ is normal so 
    \[ \frac{|KaH|}{|H|} = \frac{|KH|}{|H|} = \frac{|K|\cdot |H|}{|K \cap H|\cdot |H|} = \frac{|K|}{\gcd(|K|,|H|)} = \frac{\lcm(|K|,|H|)}{|H|} \]
    and 
    \[ |K \under G / H| = \frac{|G|}{|KH|} = \frac{|G|\cdot |K \cap H|}{|K|\cdot |H|} = \frac{|G| \cdot \gcd(|K|,|H|)}{|K|\cdot |H|} = \frac{|G|}{\lcm(|K|,|H|)}. \]
    From these observations, we conclude \eqref{eq:cyc equation} and thus that \eqref{eq:tensor square} commutes.

    Now suppose that \eqref{eq:tensor square} commutes and thus that \eqref{eq:cyc equation} holds for each subgroup $K \subseteq G$.
    Since the right hand side of \eqref{eq:cyc equation} is just a multiple of a single part with decoration $[C_{\gcd(|K|,|H|)}]$, \eqref{eq:cyc equation} implies in particular that $[C_{K \cap H}] = [C_{\gcd(|K|,|H|)}]$ since $([C_{K \cap H}], |KeH|/|H|)$ is a summand of the left hand side.
    Thus $|K \cap H| = \gcd(|K|,|H|)$ for each $K$ so that $H$ satisfies the gcd property.
\end{proof}

\bibliographystyle{alpha}
\bibliography{references.bib}

\end{document}